\documentclass[12pt]{amsart}
\usepackage{amssymb}
\usepackage{eucal}
\usepackage{amsmath}
\usepackage{amscd}
\usepackage{txfonts}      
\usepackage[dvips]{color}
\usepackage{mathtools}
\usepackage{multicol}
\usepackage[all]{xy}           
\usepackage{graphicx}
\usepackage{color}
\usepackage{colordvi}
\usepackage{xspace}
\usepackage[normalem]{ulem}
\usepackage{cancel}
\usepackage{tikz}
\usepackage{longtable}
\usepackage{multirow}
\usepackage{ifpdf}
\ifpdf
\usepackage[colorlinks,final,backref=page,hyperindex]{hyperref}
\else
\usepackage[colorlinks,final,backref=page,hyperindex,hypertex]{hyperref}
\fi
\usepackage{comment}

\newtheorem{theorem}{Theorem}[section]
\newtheorem{prop}[theorem]{Proposition}
\newtheorem{lemma}[theorem]{Lemma}
\newtheorem{coro}[theorem]{Corollary}
\newtheorem{thm-def}[theorem]{Theorem-Definition}
\newtheorem{def-prop}[theorem]{Definition-Proposition}
\newtheorem{prop-def}[theorem]{Proposition-Definition}
\newtheorem{coro-def}[theorem]{Corollary-Definition}

\theoremstyle{definition}
\newtheorem{defn}[theorem]{Definition}
\newtheorem{remark}[theorem]{Remark}

\newtheorem{exam}[theorem]{Example}

\newcommand{\nc}{\newcommand}
\nc{\tred}[1]{\textcolor{red}{#1}}
\nc{\tblue}[1]{\textcolor{blue}{#1}}
\nc{\tgreen}[1]{\textcolor{green}{#1}}
\nc{\tpurple}[1]{\textcolor{purple}{#1}}
\nc{\btred}[1]{\textcolor{red}{\bf #1}}
\nc{\btblue}[1]{\textcolor{blue}{\bf #1}}
\nc{\btgreen}[1]{\textcolor{green}{\bf #1}}
\nc{\btpurple}[1]{\textcolor{purple}{\bf #1}}

\renewcommand{\frak}{\mathfrak}
\newcommand{\efootnote}[1]{}
\renewcommand{\textbf}[1]{}
\newcommand{\delete}[1]{{}}

\delete{
\nc{\mlabel}[1]{\label{#1} {{\tt {\tiny{(#1)}}}}\ }
\nc{\mcite}[1]{\cite{#1} {{\tiny\tt (#1)}}\ }
\nc{\mref}[1]{\ref{#1}{{\tiny\tt (#1)}}\ }
\nc{\meqref}[1]{~\eqref{#1}{{\tiny\tt (#1)}}\ }
\nc{\mbibitem}[1]{\bibitem[\bf #1]{#1}}
}

	\nc{\mlabel}[1]{\label{#1}}  
	\nc{\mcite}[1]{\cite{#1}}  
	\nc{\mref}[1]{\ref{#1}}  
\nc{\meqref}[1]{~\eqref{#1}}
	\nc{\mbibitem}[1]{\bibitem{#1}} 

\nc{\sbar}{, }
\delete{
{\, {\scriptstyle{|\hspace{-.08cm}|\hspace{-.08cm}|
\hspace{-.08cm}|\hspace{-.08cm}|\hspace{-.08cm}|
\hspace{-.08cm}|\hspace{-.08cm}|\hspace{-.08cm}|\hspace{-.08cm}|
\hspace{-.08cm}|\hspace{-.08cm}|\hspace{-.08cm}|
\hspace{-.08cm}|\hspace{-.08cm}|\hspace{-.08cm}|\hspace{-.08cm}|
\hspace{-.08cm}|\hspace{-.08cm}|\hspace{-.08cm}|
\hspace{-.08cm}|\hspace{-.08cm}|\hspace{-.08cm}|\hspace{-.08cm}|
\hspace{-.08cm}|\hspace{-.08cm}|\hspace{-.08cm}|
\hspace{-.08cm}|\hspace{-.08cm}|}\, }}
}

\nc{\wvec}[2]{{\scriptsize{ \Big[\begin{array}{c} #1 \\ #2 \end{array}\Big]}}}
\nc{\lp}{\big ( }
\nc{\llp}{\Big (}
\nc{\Llp}{\left (}
\nc{\rp}{\big ) }
\nc{\rrp}{\Big )}
\nc{\Rrp}{\right )}
\nc{\lb}{\big < }
\nc{\llb}{\!\Big \langle }
\nc{\Llb}{\! \left <}
\nc{\rb}{\big >  }
\nc{\rrb}{\Big \rangle \!}
\nc{\Rb}{\Big \rangle\! }
\nc{\length}{{\rm leng}}

\nc{\id}{\mathrm{id}}

\nc{\bin}[2]{ (_{\stackrel{\scs{#1}}{\scs{#2}}})}  
\nc{\binc}[2]{ \big (\! \begin{array}{c} \scs{#1}\\
    \scs{#2} \end{array}\! \big )}  
\nc{\bincc}[2]{  \left ( {\scs{#1} \atop
    \vspace{-1cm}\scs{#2}} \right )}  
\nc{\bs}{\bar{S}}
\nc{\cosum}{\sqsubset}
\nc{\la}{\longrightarrow}
\nc{\rar}{\rightarrow}
\nc{\dar}{\downarrow}
\nc{\dap}[1]{\downarrow \rlap{$\scriptstyle{#1}$}}
\nc{\uap}[1]{\uparrow \rlap{$\scriptstyle{#1}$}}
\nc{\defeq}{\stackrel{\rm def}{=}}
\nc{\disp}[1]{\displaystyle{#1}}
\nc{\dotcup}{\ \displaystyle{\bigcup^\bullet}\ }
\nc{\gzeta}{\bar{\zeta}}
\nc{\hcm}{\ \hat{,}\ }
\nc{\hts}{\hat{\otimes}}
\nc{\barot}{{\otimes}}
\nc{\free}[1]{\bar{#1}}
\nc{\uni}[1]{\tilde{#1}}          
\nc{\hcirc}{\hat{\circ}}
\nc{\lleft}{[}
\nc{\lright}{]}
\nc{\curlyl}{\left \{ \begin{array}{c} {} \\ {} \end{array}
    \right .  \!\!\!\!\!\!\!}
\nc{\curlyr}{ \!\!\!\!\!\!\!
    \left . \begin{array}{c} {} \\ {} \end{array}
    \right \} }
\nc{\longmid}{\left | \begin{array}{c} {} \\ {} \end{array}
    \right . \!\!\!\!\!\!\!}
\nc{\ora}[1]{\stackrel{#1}{\rar}}
\nc{\ola}[1]{\stackrel{#1}{\la}}
\nc{\ot}{\otimes}
\nc{\mot}{{{\sbar}}}
\nc{\otm}{\mot}
\nc{\scs}[1]{\scriptstyle{#1}}
\nc{\subv}{{^{\star}}}
\nc{\cov}{{^{\sharp}}}
\nc{\mrm}[1]{{\rm #1}}
\nc{\dirlim}{\displaystyle{\lim_{\longrightarrow}}\,}
\nc{\invlim}{\displaystyle{\lim_{\longleftarrow}}\,}
\nc{\proofbegin}{\noindent{\bf Proof: }}
\nc{\proofend}{$\quad \square$ \vspace{0.3cm}}
\nc{\sha}{{\mbox{\cyr X}}}  
\nc{\shap}{{\mbox{\cyrs X}}} 
\nc{\shpr}{\diamond}    
\nc{\shplus}{\shpr^+}
\nc{\shprc}{\shpr_c}    
\nc{\msh}{\ast}
\nc{\vep}{\varepsilon}
\nc{\labs}{\mid\!}
\nc{\rabs}{\!\mid}

\newcommand{\Q}{\mathbb{Q}}
\newcommand{\R}{\mathbb{R}}

\newcommand{\Z}{\mathbb{Z}}

\newcommand {\cala}{{\mathcal {A}}}

\newcommand {\calh}{{\mathcal {H}}}

\newcommand {\calp}{{\mathcal {P}}}

\nc{\fraka}{{\frak a}}
\nc{\frakA}{{\frak A}}
\nc{\frakb}{{\frak b}}
\nc{\frakB}{{\frak B}}
\nc{\frakf}{{\frak F}}
\nc{\frakh}{{\frak h}}
\nc{\frakH}{{\frak H}}
\nc{\frakk}{{\frak k}}
\nc{\frakK}{{\frak K}}
\nc{\frakM}{{\frak M}}
\nc{\frakm}{{\frak m}}
\nc{\frakP}{{\frak P}}
\nc{\frakp}{{\frak p}}
\nc{\frakS}{{\frak S}}

\nc{\bfrakM}{\overline{\frakM}}

\nc {\e} {{\epsilon}}
\nc{\fpower}{\calp_{\rm fin}}
\nc{\pfpair}[2]{\Big(\begin{array}{c}\scs{#1} \\ \scs{#2} \end{array} \Big)}

\font\cyr=wncyr10
\font\cyrs=wncyr7

\newcommand {\calhd}{{{\mathcal {H}} _{\Z _{\ge 1}}}}

\newcommand {\rmid} {{\rm Id}}
\newcommand {\fch} {{F^{\mathrm{Ch}}}}

\newcommand {\p} {{\partial}}
\nc{\dd}{p}
\nc{\dep}{\mathrm{dep}}

\nc{\shapt}{\overline{\shap}}

\nc{\emzvsha}{MZV shuffle algebra\xspace} 

\newcommand {\reccop} {{\Delta_{\ge 1}}}	
\newcommand {\descop} {\Delta'}				
\newcommand{\fraccop}{\Delta^{\mathrm{chen}}}	

\nc {\cks}{\text{\textcircled {s}}\xspace}
\nc{\mzvalg}{\mathbf{MZV}}

\nc{\zb}[1]{\textcolor{green}{Bin: #1}}
\nc{\xhy}[1]{\textcolor{blue}{Yu: #1}}
\nc{\li}[1]{\textcolor{purple}{#1}}
\nc{\lir}[1]{\textcolor{purple}{Li: #1}}
		
\begin{document}

\title[Hopf algebra for MZV shuffle algebra]{Hopf algebras for the shuffle algebra and fractions from multiple zeta values}
%

\author{Li Guo}
\address{Department of Mathematics and Computer Science, Rutgers University at Newark, Newark, New Jersey, 07102, United States}

\author{Wenchuan Hu}
\address{School of Mathematics,
Sichuan University, Chengdu, 610064, P. R. China}
\email{huwenchuan@gmail.com}

\author{Hongyu Xiang}
\address{School of Mathematics,
Sichuan University, Chengdu, 610064, P. R. China}
\email{xianghongyu1@stu.scu.edu.cn}

\author{Bin Zhang}
\address{School of Mathematics,
Sichuan University, Chengdu, 610064, P. R. China}
\email{zhangbin@scu.edu.cn}

\date{\today}

\begin{abstract} The algebra of multiple zeta values (MZVs) is encoded as a stuffle (quasi-shuffle) algebra and a shuffle algebra. The MZV stuffle algebra has a natural Hopf algebra structure which has important applications to MZVs. This paper equips a Hopf algebra structure to the MZV shuffle algebra. The needed coproduct is defined by a recursion through a family of weight-increasing linear operators. To verify the Hopf algebra axioms, we make use of a family of fractions, called Chen fractions, that have been used to study MZVs and also serve as the function model for the MZV shuffle algebra. Applying natural derivations on functions and working in the context of locality, a locality Hopf algebra structure is established on the linear span of Chen fractions. This locality Hopf algebra is then shown to descend to a Hopf algebra on the MZV shuffle algebra, whose coproduct satisfies the same recursion as the first-defined coproduct. Thus the two coproducts coincide, establishing the needed Hopf algebra axioms on the MZV shuffle algebra.
\end{abstract}

\subjclass[2020]{
11M32,	
16T05,	
12H05, 
16T30,  
16W25,	
16S10,	
40B05	
}

\keywords{multiple zeta value, shuffle algebra, Chen  fraction, locality algebra, Hopf algebra, locality}

\maketitle

\vspace{-1cm}

\tableofcontents

\vspace{-1cm}

\setcounter{section}{0}


\section{Introduction}
This paper constructs a Hopf algebra structure on the shuffle algebra from multiple zeta values, where the coproduct satisfies a differential type recursion and is the descendant of a natural coproduct on the space of fractions defining MZVs.

\subsection{Shuffle and quasi-shuffle algebras for MZVs}

Multiple zeta values (MZVs) are the evaluations of the multiple zeta series
\begin {equation}
\mlabel {eq:zeta}
\zeta (s_1, \cdots, s_k)=\sum _{n_1>\cdots > n_k>0}\frac 1{n_1^{s_1}\cdots n_k^{s_k}}
\end{equation}
at positive integer arguments such that the series converges. This means that $s_1>1$, $s_i\in \Z _{\ge 1}, i=2, \cdots, k$. MZVs  and  their  generalizations  have been studied  extensively  from  different  viewpoints since the early 1990s with  connections  to  number theory, algebraic  geometry,  mathematical  physics,  quantum  groups  and  knot  theory \mcite{BBV, BBBL,  BK,Brown, Car, Gon, GM, GZ1, H3, Kre, Ter, Zh}.

Let
$$
\mzvalg: = \Q 1+ \Q\big\{\zeta(s_1,\cdots,s_k)\ \big|\ s_1\geq 2, s_i\geq 1, i\geq 2\,\big\} \subseteq \R
$$
denote the subspace of $\R$ spanned by MZVs and $1$.
A fascinating aspect of their study is the rich algebraic relations among these analytically defined values, especially the stuffle (or quasi-shuffle) relation and the shuffle relation.

Through the encoding of an MZV $\zeta(\vec s)$ by a basis element $[\vec s]$ in the vector space
\begin {equation}
\mlabel {eq:cshafalg}
\calh ^0:=\Q {\bf1} \oplus \bigoplus_{k\in \Z_{>0}, \vec s\in \Z_{\ge 1}^k, s_1>1}\Q[\vec s],
\end{equation}
the stuffle relation of MZVs is interpreted as the algebra homomorphism
\begin{equation}
\zeta_\ast: (\calh^0, *)\to \mzvalg, \quad [\vec{s}]\mapsto \zeta(\vec s).
\mlabel{eq:stufhom}
\end{equation}

On the other hand, an MZV $\zeta(\vec s)$ is expressed as the integral (often named after Chen or Drinfeld or Kontsevich)~\mcite{HO,Zag}
{\small\begin {equation}
	\mlabel {eq:zetai}
	\zeta (s_1, \cdots s_k)=\underbrace {\int_ 0^1\frac {dt}t\int _0^t \frac {dt}t\cdots \int _0^t\frac {dt}t}_{s_1-1}\int _0^t\frac {dt}{1-t}\cdots \underbrace {\int _0^t\frac {dt}t\cdots \int _0^t\frac {dt}t}_{s_k-1}\int _0^t\frac {dt}{1-t}
\end{equation}}
With the encoding of $\zeta(s_1,\ldots,s_k)$  by $x_0^{s_1-1}x_1\cdots x_0^{s_k-1}x_1\in \Q\langle x_0,x_1\rangle$, the space of MZVs is encoded by the subspace $x_0\Q\langle x_0,x_1\rangle x_1\subseteq \Q\langle x_0,x_1\rangle$. Equipping the latter with the shuffle product $\shap$,
the shuffle relation of MZVs is interpreted as the algebra homomorphism
\begin{equation}
	\zeta_\shap: (x_0\Q\langle x_0,x_1\rangle x_1,\shap)\to \mzvalg, \quad
x_0^{s_1-1}x_1\cdots x_0^{s_k-1}x_1\mapsto \zeta(s_1,\ldots,z_k).
\mlabel{eq:shufhom}
\end{equation}

These two encodings of the MZVs are integrated by the linear isomorphism
\begin {equation}
\mlabel {eq:rho}
\rho: \calh^0 \to x_0\Q \langle x_0, x_1 \rangle x_1, \ [s_1, \cdots, s_k]\to x_0^{s_1-1}x_1\cdots x_0^{s_k-1}x_1,
\end{equation}
through which the shuffle product $\shap$ on $x_0\Q\langle x_0,x_1\rangle x_1$ is pulled back to a multiplication $\shapt$ on $\calh^0$,
completing the commutative diagram of linear maps
\begin{equation}
	\begin{split}
\xymatrix{ (\calh^0,\ast,\shapt) \ar^{\rho}[rr] \ar_{\zeta_\ast}[rd] &&
	(x_0\Q\langle x_0,x_1\rangle x_1, \shap) \ar^{\zeta_\shap}[ld]\\
	& \mzvalg & }
\end{split}
\mlabel{eq:sthhom}
\end{equation}
in which $\zeta_*$ and $\zeta_\shap$ are algebra homomorphisms, and  $\rho:(\calhd,\shapt)\to (x_0\Q\langle x_0,x_1\rangle x_1, \shap)$ is an algebra isomorphism.

Consequently,
$$\big \{\zeta_* ([\vec s] *[\vec t])-\zeta_* ([\vec s] \shapt [\vec t])=0\ \big| \ [\vec s], [\vec t]\in \calh^0 \big\}
$$
is a family of $\Q$-linear relations among MZVs, called the {\bf double shuffle relation}.

Denote
\begin{equation}
	\calhd:= \Q {\bf1} \oplus \bigoplus_{k\in \Z_{>0}, \vec s\in \Z_{\ge 1}^k}\Q[\vec s].
\mlabel{eq:hgeq0}
\end{equation}
Ihara, Kaneko and Zagier \mcite {IKZ} extended the commutative diagram in Eq.~\meqref{eq:sthhom} to a commutative diagram
\begin{equation}
	\begin{split}
		\xymatrix{ (\calhd,\ast,\shapt) \ar^{\rho}[rr] \ar_{\zeta_\ast}[rd] &&
			(\Q\langle x_0,x_1\rangle x_1,\shap) \ar^{\zeta_\shap}[ld]\\
			& \mzvalg[T] & }
	\end{split}
	\mlabel{eq:sthhom2}
\end{equation}
and extended the above double shuffle relation to the {\bf extended double shuffle relation}~\mcite{IKZ,Rac}
\begin{equation}
	\big\{ \zeta_*([\vec s]\ast [\vec t]) - \zeta_*([\vec s] \shapt [\vec t]),\ \zeta_*([1] \ast [\vec t]) - \zeta_*([1] \shapt [\vec t])\ |\
	[\vec s], [\vec t]\in\calh^0\big\},
	\mlabel{eq:edsval}
\end{equation}
where $[1]$ is the base element of $\calhd$ corresponding to $1\in\Z_{\ge 1}$.
\begin{theorem} {\bf (\mcite{IKZ,Rac})}
	Let $I_{\mathrm{EDS}}$ be the ideal of $\calhd$ generated by the set
\begin{equation}
	\big\{ [\vec s]\ast [\vec t] - [\vec s] \shapt [\vec t],\ [1] \ast [\vec t] - [1] \shapt [\vec t]\ \big|\
[\vec s], [\vec t]\in\calh^0\big\}.
	\mlabel{eq:eds}
\end{equation}
Then $I_{\mathrm{EDS}}$ is in the kernel of $\zeta_*$.
	\mlabel{thm:eds}
\end{theorem}
It is conjectured that $I_{\mathrm{EDS}}$ is in fact the kernel of $\zeta_*$.
With this connection with MZVs, we will call $(\calhd,*)$ and $(\calhd,\shapt)\cong (\Q\langle x_0,x_1\rangle x_1,\shap)$ the {\bf MZV quasi-shuffle algebra} and {\bf \emzvsha} respectively.

\subsection{Hopf algebra structures on the MZV quasi-shuffle algebra and \emzvsha}

Further understanding of MZVs depends on revealing deeper structures of the MZV quasi-shuffle algebra $(\calhd,*)$ and MZV shuffle algebra $(\calhd,\shapt)$, including their possible Hopf algebra structures.

By the general construction, the MZV quasi-shuffle algebra has a natural enrichment to a Hopf algebra with the deconcatenation coproduct~\mcite{H2}.
Its action on the MZV shuffle algebra has been used to obtain large classes of algebraic relations of MZVs~\mcite{HO,Oh}.
Hopf algebras have played a critical role in the study of motivic MZVs~\mcite{Brown,GF,Rac}.

The situation is quite different for a Hopf algebra structure on the \emzvsha $(\calhd,\shapt)$. 
First, the space $\calhd$ has the natural deconcatenation of the vectors $[s_1,\ldots,s_k]\in \Z^k_{\geq 0}$:
$$ [s_1,\ldots,s_k] \mapsto \sum_{i=0}^k [s_1,\ldots,s_i]\ot [s_{i+1},\ldots,s_k].$$ 
But it is not compatible with the product $\shapt$ on $\calhd$. 

Next note that in the \emzvsha $(\calhd,\shapt)$, the notion of $[\vec{s}]\in \calhd$ is a contracted form of elements in $\Q\langle x_0,x_1\rangle x_1$ via the algebra isomorphism $\rho$ extending Eq.~\meqref{eq:rho}:
\begin {equation}
\mlabel {eq:rho2}
\rho: (\calhd,\shapt) \to (\Q \langle x_0, x_1 \rangle x_1,\shap), \ [s_1, \cdots, s_k]\to x_0^{s_1-1}x_1\cdots x_0^{s_k-1}x_1. 
\end{equation}
%
The shuffle algebra $(\Q \langle x_0,x_1\rangle x_1,\shap)$ has a natural Hopf algebra structure with the deconcatenation coproduct of words. However, this coproduct does not pull back to $(\calhd,\shapt)$. As a simple example, for the word $x_0^2x_1\in \Q\langle x_0,x_1\rangle$ corresponding to $[3]\in \calhd$, its deconcatenation coproduct is
$$ x_0^2x_1\ot 1 + x_0^2\ot x_1+x_0\ot x_0x_1 +1\ot x_0^2x_1$$
in which $x_0\ot x_0x_1$ and $x_0^2\ot x_1$ are not in $\Q\langle x_0,x_1\rangle x_1$ and hence do not correspond to elements in $\calhd$.


{\it The purpose of this paper is to provide a Hopf algebra structure on $(\calhd,\shapt)$.} 
In fact, the definition of the coproduct is motivated by a natural coproduct of Chen fractions which define MZVs. Further applications will be given in later studies~\mcite{GXZ}.

The rest of the introduction will convey some idea of our approach and give an outline of the construction.

\subsection{New Hopf algebra structure on the MZV shuffle algebra}

The goal of Section~\mref{s:main} is to introduce just enough notions to state the main theorem on the Hopf algebra structure on $(\calhd,\shapt)$. The main ingredient of the construction is the coproduct $\reccop$ which can be most easily defined by a recursion derived from a coderivation condition on a family of linear operators on $\calhd$ (Definition \mref{d:reccop}).
Examples are provided to illustrate the recursion. The main result on the corresponding Hopf algebra structure is stated in Theorem~\mref{thm:shuhopf}, followed by a sketch of the proof, and the verification that the coproduct is well defined (Proposition~\mref{p:uniquecoproduct}).

It is might be possible to prove Theorem~\mref{thm:shuhopf} by directly verifying the axioms of a bialgebra. However, checking that the coproduct is an algebra homomorphism is made challenging due to the complexity of the shuffle product in $\calhd$. For example, the shuffle product of two one-dimensional vectors $[s]$ and $[t]$ is the Euler decomposition formula in disguise \mcite{BBG}:
$$ [s]\shapt [t]=\sum _{i=0}^{s-1} \binc{s+i-1}{i} [s+i,t-i]
+ \sum_{j=0}^{t-1} \binc{t+j-1}{j} [t+j,s-j].
$$
In general, a product formula is provided in~\mcite{GX2}, but its complexity makes it impractical for further computations.

We will instead take an indirect yet quite natural approach, using tools from Chen fractions and locality. This is carried out in Section~\mref{s:locchen}. The first tool is the representation of MZVs by functions of the form
$$\frac{1}{(x_{i_1}+x_{i_2}+\cdots+x_{i_k})^{s_1}(x_{i_2}+\cdots+x_{i_k})^{s_2}\cdots x_{i_k}^{s_k}},
$$
where $s_1, \cdots, s_k\in \Z _{\ge 1}$. They are called Chen fractions due their relation to Chen cones~\mcite{GPZ2}.
They are also called MZV fractions~\mcite{GX} since they define MZVs~\mcite{BBBL,H1,LM}:
$$ \zeta(s_1,\ldots,s_k)=\sum_{x_{i_1},\ldots,x_{i_k}\in \Z_{\geq 1}} \frac{1}{(x_{i_1}+x_{i_2}+\cdots+x_{i_k})^{s_1}(x_{i_2}+\cdots+x_{i_k})^{s_2}\cdots x_{i_k}^{s_k}}.
$$
The product of two Chen fractions with disjoint variables satisfies the shuffle relation, giving rise to the shuffle relation of MZVs~\mcite{GPZ2,GX}, thus serving the same purpose as the integral representations of MZVs in Eq.~\meqref{eq:zetai}.
Therefore, a Hopf algebraic structure on Chen fractions is of independent interest and is the motivation for our choice of the Hopf algebra structure on $\calhd$.

Working with the function subspace spanned by Chen fractions has the advantage they are naturally equipped with a family of partial derivatives. Since these derivations increase the degree of a Chen fraction, they can be used to give a recursive definition of a coproduct, by requiring that the derivations are also coderivations with respect to the tensor product (Definition~\mref{d:chencop}).
Furthermore, as the product of two Chen fractions is just the function multiplication, the derivations can be used give an inductive argument to verify the bialgebra axioms.

However, in order for these advantages of Chen fractions to take effect, their space first needs to be closed under the function multiplication, which is not the case. So apply our second tool of locality algebraic structures, motivated by the locality principle in physics contexts and abstracted into an algebraic framework in~\mcite{CGPZ,GPZ2,GPZ1,GPZ3}. A locality set is simply a set equipped with a symmetric binary relation. Imposing compatibility conditions of the relation with various algebraic axioms leads to the corresponding locality algebraic structures, in particular locality Hopf algebras. Applying the locality framework to the space Chen fractions equips the space with a locality Hopf algebra structure (Theorem~\mref{t:fchDeltalocalhopf}). In fact, with the natural derivations, it is a locality multi-differential Hopf algebra, generalizing the classical notion of differential Hopf algebras~\mcite{Ar,Br} as well as providing a locality setting for the recent study of multi-differential algebra and multi-Novikov algebra arising from regularity structures in stochastic PDEs~\mcite{BD,BHZ}.

In Section~\mref{s:mzvhopf}, we show that, through the natural projection $$\frac{1}{(x_{i_1}+x_{i_2}+\cdots+x_{i_k})^{s_1}(x_{i_2}+\cdots+x_{i_k})^{s_2}\cdots x_{i_k}^{s_k}} \mapsto [s_1,\ldots,s_k]
$$
from the space of Chen fractions to $\calhd$, the locality Hopf algebra of Chen fractions descends to a Hopf algebra structure on $(\calhd,\shapt)$ (Theorem~\mref{t:fchDeltalocalhopf}). We then show that the coproduct $\descop$ of the descended Hopf algebra satisfies the same recursion that defines the coproduct $\reccop$ introduced at the beginning of the paper. Therefore the two coproducts agree, showing that $\reccop$ equips $\calhd$ with the desired Hopf algebra structure.

\smallskip

\noindent
{\bf Notations. } For $[\vec s]\in \calhd$ with $\vec s\in \Z^k_{\ge 1}$ denote $|[\vec s]|=|\vec s|=s_1+\cdots+s_k$ for the weight and $\dep(\vec s):=k$ for the {\bf depth}, with the convention that $\Z^0_{\ge 1}=\{{\bf1}\}$ and $|{\bf 1}|=0$ and $\dep(\bf1)=0$.

\section {Statement of the main theorem}
\mlabel{s:main}
This section first introduces a family of linear operators on the \emzvsha $(\calhd,\shapt)$. Requiring that these linear operators behave like a shifted coderivations gives rise to a recursion that defines a comultiplication on $\calhd$. We then state our main theorem that the \emzvsha
with this comultiplication is a Hopf algebra and give an outline of the proof.

\subsection{Statement of the main theorem and an outline of the proof}
\mlabel{ss:mainthm}

With the notation in Eq.\meqref{eq:hgeq0}, for $i\geq 1$, define the linear map
{\small
\begin{equation}
\begin{split}
	\delta_i: \calhd& \longrightarrow \calhd,\\
{\bf1}&\mapsto 0, \\
	[s_1,\cdots,s_k]& \mapsto \left\{\begin{array}{ll}
			\sum_{j=1}^i s_j[s_1,\cdots, s_j+1,\cdots, s_i,\cdots,s_k], & i\le k,  \\
			0,& i>k.
		\end{array}\right.
\end{split}
\mlabel{eq:deltadef}
\end{equation}
}
By convention, we take $\delta_i=0$  when  $i\le 0$.
For example,
$$ \delta_1([2,1])=2[3,1], \quad \delta_2([2,1])=2[3,1]+[2,2], \quad
\delta_3([2,1])=0.$$

\begin{remark}
Unlike the derivations have appeared in the study of MZVs~\mcite{HO,IKZ}, the linear maps $\delta_i$ are not derivations, even though are modeled after the partial derivatives on fractions in Eq.~\meqref{eq:partial}.
As a simple example,
$$\delta_1([1]\shapt [1]) =\delta_1([2[1,1]])=2[2,1].$$
But
$$\delta_1([1])\shapt [1]+[1]\shapt\delta_1([1])=[2]\shapt[1]+[1]\shapt[2]=4[2,1]+2[1,2]\neq \delta_1([1]\shapt [1]).
$$
It would be interesting to explore their relations.
\end{remark}

Denote
\begin{equation}
	\dd_i:=\delta_i-\delta_{i-1}, \quad  i\geq 1.
	\mlabel{eq:pij}
\end{equation}
We display the following basic properties of these operators for later use.
\begin{lemma} The following equalities hold.
\begin{align}
\dd_{i} [s_1,\ldots,s_k]&=s_i [s_1,\ldots,s_{i-1},s_i+1,s_{i+1},\ldots,s_k], \quad 1\leq i\leq k,
\mlabel{eq:ptos}\\
\dd_{k+1} [s_1,\ldots,s_k]&=-\sum_{j=1}^k s_j[s_1,\cdots, s_j+1,\cdots, s_k],\mlabel{eq:drecur} \\
\dd_{i} [s_1,\ldots,s_k]&=0, \quad i\leq 0 \text{ or } i\geq k+2,
\mlabel{eq:ptos3}\\
\dd_{i}\dd_{j}&=\dd_{j}\dd_{i}, \quad i, j\in \Z,
\mlabel{eq:pcomm}\\
\delta_i&=\sum_{j=1}^i \dd_{j}, \quad  i\in \Z, \mlabel{eq:dtop}\\
	\delta_i \delta_j&=\delta_j \delta_i, \quad i, j\in \Z.
\mlabel{eq:dcomm}\\
	[s_1, s_2,\ldots,s_k ]
	&=\frac{\dd_{1}^{s_1-1} \dd_{2}^{s_2-1}\cdots  \dd_{k}^{s_k-1}}{(s_1-1)!(s_2-1)!\cdots (s_k-1)!}[1,1,\cdots,1]. \mlabel{eq:dalggen}
\end{align}
\mlabel{l:comm}
\end{lemma}

\begin{proof}
Eqs.~\meqref{eq:ptos} -- \meqref{eq:ptos3} follow from the definitions.

Eqs.~\meqref{eq:pcomm} and \meqref{eq:dtop} follow from Eqs.~\meqref{eq:ptos} -- \meqref{eq:ptos3}.

Eq.~\meqref{eq:dcomm} follows from Eqs.~\meqref{eq:pcomm} and \meqref{eq:dtop}.

Eq.~\meqref{eq:dalggen} follows from applying Eq.~\meqref{eq:ptos} repeatedly.
\end{proof}

For a linear operator $A$ on $\calhd$ and a sequence of linear operators $f_i, i\geq 0$ on $\calhd$, we define the {\bf shifted tensor}
\begin{equation}
\begin{split}
A\cks f_i: \calhd\otimes \calhd&\longrightarrow \calhd\otimes \calhd, \\
[\vec s]\otimes [\vec t]&\mapsto
	A([\vec s])\otimes f_{i-\dep(\vec s)}([\vec t]).
\end{split}
	\mlabel{eq:shiftten}
\end{equation}

We now introduce the main notion for this study.
\begin{defn}
	\mlabel{d:reccop}
Define a linear map
\begin{equation}
\reccop : \calhd\longrightarrow \calhd\otimes \calhd
\mlabel{eq:recop}
\end{equation}
by the following recursion.
	\begin{enumerate}
		\item  $\reccop({\bf 1}):={\bf 1}\otimes {\bf 1}$;
		\item $\reccop([\{1\}^k]):=\sum\limits_{j=0}^{k}[\{1\}^j]\otimes [\{1\}^{k-j}]$. Here we use the abbreviation $[\{1\}^k]=[\underbrace{1,\ldots,1}_{k}]$;
\item \mlabel{it:reccop3}
For $s_i\geq 1$ and $1\leq i\leq k$, define
\begin{equation}
\mlabel{eq:reccop3}
\begin{split}
\reccop([s_1,\ldots,s_{i-1},s_i+1,s_{i+1},\ldots,s_k])
	:=\frac{1}{s_i} \Big(\id \cks \dd_i+\dd_i\ot \id\Big) \reccop([s_1,\ldots,s_k]).
\end{split}
\end{equation}
	\end{enumerate}
\end{defn}
We note that iterations of the recursion in Eq.~\meqref{eq:reccop3} involves the choice of $i$. The fact that it gives rise to a well-defined map $\reccop$ will be proved in Proposition~\mref{p:uniquecoproduct}.

The third condition amounts to requiring that $\delta_i$ is a {\bf shifted coderivation} in the sense that
\begin{equation}
\reccop \delta_i=(\id\cks \delta_i+\delta_i\otimes \id) \reccop, \quad i\geq 0.
	\mlabel{eq:dalgrec}
\end{equation}
Equivalently,
\begin{equation}
	\reccop \dd_i=(\id\cks \dd_i+\dd_i\otimes \id) \reccop, \quad i\geq 0.
	\mlabel{eq:dalgrec2}
\end{equation}

\begin{remark}
We recall a commonly used notation for the differentials in homological algebra, differential graded algebras, in particular differential Hopf algebras~\mcite{Ar,Br,Sch} which we will generalized to the locality setting later in the paper (Definition~\mref{de:ldha} and Theorem~\mref{t:fchDeltalocalhopf}).
Let $d:A\to A$ be a derivation on a graded algebra $A$, often further requiring $d^2=0$. Denoting
$$(\omega \otimes d)(a\otimes b):=\omega(a)\otimes d(b)=(-1)^{\mathrm{deg} (a)}a\otimes d(b)$$
leads to the basic notion of {\bf tensor product differential} which is denote by $\bar{d}$ in~\mcite{Ar} and $d_\ot$ in \mcite{Sch}:
$$ \bar{d}:=d_\otimes:=\omega \ot d+d\ot \id.$$
Now denote $\delta_i(a):=(-1)^{i}d(a), i\in \Z$. Then our shifted tensor $\id \cks \delta_0$ recovers $\omega \ot d$:
$$ (\id \cks \delta_0)(a\ot b) =a\ot \delta_{-\mathrm{\deg}(a)}(b)
=a\ot (-1)^{-\mathrm{\deg}(a)}d(b)=(-1)^{\mathrm{\deg}(a)}a \ot d(b)
= \omega(a)\ot d(b).$$
Consequently, $ \id \cks \delta_0+\delta_0\otimes \id$ in Eq.~\meqref{eq:dalgrec} recovers $\bar{d}$ and $d_\ot$.
	\mlabel{rk:twist}
\end{remark}

We give some examples to illustrate the recursion in Eq.~\meqref{eq:reccop3} that defines $\reccop$.
Some values of the coproduct have the well-recognized pattern of the deconcatenation coproduct, such as
\begin{align*}
\reccop([1,1])&={\bf 1}\otimes [1,1]+[1]\otimes [1]+[1,1]\otimes {\bf 1}, \\
\reccop([2,1])&={\bf 1}\otimes [2, 1]+[2]\otimes [1]+[2,1]\otimes {\bf 1},\\
\reccop([3,1])&={\bf 1}\otimes [3,1]+[3]\otimes [1]+[3,1]\otimes {\bf 1},\\
\reccop([2,1,1])&={\bf 1}\otimes [2,1,1]+[2]\otimes [1,1]+[2,1]\otimes [1]+[2,1,1]\otimes {\bf 1}.
\end{align*}
Most other values have extra terms. For example,
\begin{align*}
	\reccop([1,2])&={\bf 1}\otimes [1,2]+[1]\otimes [2]-[2]\otimes [1]+[1,2]\otimes {\bf 1}, \\
\reccop([2,2])&={\bf 1}\otimes [2, 2]+[2]\otimes [2]-2[3]\otimes [1]+[2,2]\otimes {\bf 1}, \\
\reccop([1,3])&={\bf 1}\otimes [1,3]+[1]\otimes [3]-[2]\otimes [2]+[3]\otimes [1]+[1,3]\otimes {\bf 1},\\
	\reccop([1,2,1])&={\bf 1}\otimes [1,2,1]+[1]\otimes [2,1]-[2]\otimes [1,1]+[1,2]\otimes [1]+[1,2,1]\otimes {\bf 1},\\
	\reccop([1,1,2])&={\bf 1}\otimes [1,1,2]+[1]\otimes [1,2]+[1,1]\otimes [2]-[2,1]\otimes [1]-[1,2]\otimes [1]+[1,1,2]\otimes {\bf 1}.
\end{align*}

To see how the recursion applies, we give details for one example.
\begin{align*}
\reccop([1,2])=&\reccop (\dd_2 [1,1])\\
=&\big(\id\cks \dd_2+\dd_2\otimes \id \big)\reccop([1,1])\\
=&\big(\id\cks \dd_2+\dd_2\otimes \id \big)({\bf 1}\otimes [1,1]+[1]\otimes [1]+[1,1]\otimes {\bf 1})\\
=&{\bf 1}\otimes \dd_2([1,1])+[1]\otimes \dd_1([1])+[1,1]\otimes \dd_0({\bf 1})\\
&+\dd_2({\bf 1})\otimes [1,1]+\dd_2([1])\otimes [1]+\dd_2([1,1])\otimes {\bf 1}\\
=&{\bf 1}\otimes [1,2]+[1]\otimes [2]-[2]\otimes [1]+[1,2]\otimes {\bf 1}.
\end{align*}

It would be interesting to find possible implications of these coproducts to MZVs, especially the terms with negative coefficients.

Define a linear map
\begin{equation}
	\varepsilon_{\geq 1}: \calhd\longrightarrow \Q, \quad
	\left\{\begin{array}{l} \varepsilon_{\geq 1}({\bf1}):=1, \\
		\varepsilon_{\geq 1}([\vec s]):=0, \vec s\in \Z_{\ge 1}^k,\  k\ge 1.
	\end{array}
	\right.
	\mlabel{eq:counit}
\end{equation}
Here is our main theorem in this study.
\begin{theorem}
\begin{enumerate}
\item The recursion for $\reccop$ in Definition~\mref{d:reccop} is well defined and unique.
\mlabel{i:shuhopf1}
\item	The quintuple $(\calhd, \shapt, \reccop, {\bf1}, \vep_{\ge 1})$  is a connected bialgebra and hence a Hopf algebra.
\mlabel{i:shuhopf2}
\end{enumerate}
\mlabel{thm:shuhopf}
\end{theorem}

\begin{proof} [Outline of the proof of Theorem~\mref{thm:shuhopf}]
The proof of Item \meqref{i:shuhopf1} will be given in Proposition~\mref{p:uniquecoproduct}.

The proof of Item~\meqref{i:shuhopf2} is divided into the following three steps.
\begin{enumerate}
	\item[{\bf Step 1.}] Construct a locality Hopf algebra on Chen fractions (Theorem~\mref{t:fchDeltalocalhopf}). It is in fact a locality differential Hopf algebra;
	\item[{\bf Step 2.}] Show that this locality Hopf algebra descends to a Hopf algebra structure on $\calhd$ (Theorem~\mref{t:deschopf});
	\item[{\bf Step 3.}] Show that the coproduct $\descop$ of the descended Hopf algebra structure on $\calhd$ satisfies the same recursion that defines the coproduct $\reccop$ in Definition~\mref{d:reccop} (Proposition \mref{p:copcompatible}). Hence $\descop$ coincides with the coproduct $\reccop$. Therefore, the structure in Theorem~\mref{thm:shuhopf}.\meqref{i:shuhopf2} is also a Hopf algebra
\end{enumerate}
This completes the proof of Theorem~\mref{thm:shuhopf}.
\end{proof}

\subsection{Well-definedness and uniqueness of the coproduct}
\mlabel{ss:unique}
The purpose of this subsection is to prove Theorem~\mref{thm:shuhopf}.\meqref{i:shuhopf1}, that the coproduct $\reccop$ is well defined and unique.

\begin{prop}
The recursion for $\reccop$ in Definition \mref{d:reccop} is well defined and unique.
\mlabel{p:uniquecoproduct}
\end{prop}

\begin{proof}
Possible ambiguities in defining the value $\reccop([\vec s])$ arises from the choice of $1\leq i\leq k$ in applying the recursion~\meqref{eq:reccop3}.	More precisely, applying Eqs.~\meqref{eq:dalggen} and Eq.(\ref{eq:dalgrec2}), we obtain
\begin{align*}
 \reccop([s_1,\ldots,s_k])& =
 \reccop\Big(\prod_{i=1}^k\frac{\dd_{i}^{s_i-1}}{(s_i-1)!}([1,1,\cdots,1])\Big)\\
 &= \Big(\prod_{i=1}^k \frac{(\id \cks \dd_i+\dd_i\ot \id)^{s_i-1}}{(s_i-1)!}\Big)\reccop ([1,\ldots,1]).
\end{align*}
We just need to show that the end result does not depend on the order in which the operators $\id \cks \dd_i+\dd_i\ot \id, 1\leq i\leq k,$ are applied. This amounts to the commutativity
$$ (\id \cks \dd_i+\dd_i\ot \id)(\id \cks \dd_j+\dd_j\ot \id)=
(\id \cks \dd_j+\dd_j\ot \id)(\id \cks \dd_i+\dd_i\ot \id), \quad 1\leq i, j\leq k,$$
which distributes to
$$  (\id \cks \dd_i)(\id \cks \dd_j)=(\id \cks \dd_j)(\id \cks \dd_i), \quad (\id \cks \dd_i)(\dd_j\ot \id)=(\dd_j\ot \id)(\id \cks \dd_i), $$
$$ (\dd_i\ot \id)(\id \cks \dd_j)=(\id \cks \dd_j)(\dd_i\ot \id), \quad (\dd_i\ot \id)(\dd_j\ot \id)=(\dd_j\ot \id)(\dd_i\ot \id), \quad 1\leq i, j, \leq k.$$
These are all easily verified by Eq.(\mref{eq:pcomm}). The uniqueness follows from the recursion.
\end{proof}

\section{Locality Hopf algebra of Chen fractions}
\mlabel{s:locchen}
In this section, we obtain a Hopf algebra structure on the space of Chen fractions in the locality framework~\mcite{CGPZ,GPZ1}. It is of independent interest because of the close relationship of the Chen fractions and MZVs.

\subsection{Locality multi-differential algebra of Chen fractions}
\mlabel{subsec:localgerm}

We first recall needed notions on locality algebraic structures ~\mcite{CGPZ,GPZ3}, culminating in the new concept of locality multi-differential algebras. With each notion, we will use Chen fraction as the primary example which will also be the main application.

To begin with, a {\bf locality set} is a  couple $(X, \top)$ where $X$ is a set and
	$$ \top:= X\times_\top X \subseteq X\times X$$
	is a binary symmetric relation, called a {\bf locality relation} on $X$. For $x_1, x_2\in X$,
	denote $x_1\top x_2$ if $(x_1,x_2)\in \top$. For a subset $U\subset X$, define the {\bf  {polar} subset} of $U$ by
\begin{equation*}
U^\top:= \{x\in X\,|\, (x,U)\subseteq \top \}.  			
\end{equation*}

For a variable set
$\{x_i\}_{i\in \Z _{\ge 1}}$, and a finite nonempty subset $I=\{i_1, \cdots, i_k\} \subset \Z _{\ge 1}$, recall that a {\bf Chen fraction} \mcite{GPZ1} with variables in  $\{x_i\ | i\in I\}$ is
a fraction of the form
$$\wvec{s_1, \cdots, s_k}{x_{i_1},\cdots, x_{i_k}}:=\frac{1}{(x_{i_1}+x_{i_2}+\cdots+x_{i_k})^{s_1}(x_{i_2}+\cdots+x_{i_k})^{s_2}\cdots x_{i_k}^{s_k}},
$$
where $s_1, \cdots, s_k\in \Z _{\ge 1}$.
The fraction is also called a MZV fraction~\mcite{GX} because its free summation for the variable $x_i$ over $\Z_{\ge 1}$ gives the MZV:
\begin{equation}
\zeta(s_1,\ldots,s_k)=\sum_{n_1,\ldots,n_k=1}^\infty
\frac{1}{(n_1+n_2+\cdots+n_k)^{s_1}(n_2+\cdots+n_k)^{s_2}\cdots n_k^{s_k}}.
\end{equation}
See~\mcite{BBBL,H1,LM} for further studies of MZVs with this approach.

Let $\fch$ denote the  set of  all Chen fractions with variables in $\{x_i\}_{i\in \Z _{\ge 1}}$, together with the constant function $1$.
Following \mcite {GPZ1},
define a locality relation on $\fch$ by
$$f_1\top  f_2 \, \text{if } f_1, f_2\in \fch \text{ are in disjoint subsets of variables},
$$
and $1\top f$ for arbitrary Chen fraction $f$.

Further recall that a {\bf locality vector space} is a vector space $V$ equipped with a locality relation $\top$ such that, for each subset $X$ of $V$, $X^\top$ is a linear subspace of $V$.
Moreover, a {\bf locality algebra} over a field $K$ is a locality vector space $(A,\top)$ over  a field $K$ together with a bilinear map
$$ m_A: A\times_\top A \to A,
(u,v)\mapsto  u\cdot v=m_A(x,y)$$
satisfying the following axioms:
\begin{enumerate}
	\item[(a)] For $u, v, w\in A$ with $u\top v, u\top w, v\top w$, there is
	\begin{equation*}
		(u\cdot v) \top w,\quad  u\top (v\cdot w), \quad (u\cdot v) \cdot w = u\cdot (v\cdot w). 	
		\mlabel{eq:asso}
	\end{equation*}
	\item[(b)] For $u, v, w\in A$ with $u\top w, v\top w$,
	$$ (u+v)\cdot w = u\cdot w + v\cdot w, \quad   w\cdot (u+v) = w\cdot u+w\cdot v, $$
	$$ (ku)\cdot w =k(u\cdot w),\quad  u\cdot (kw)=k(u\cdot w), \ k\in K.$$
	\item[(c)] There is a unit $1_A$ such that $1_A\top u$ for each $u\in A$, and
	$$1_A\cdot u=u\cdot 1_A=u.$$
\end{enumerate}

Continuing with the above example of locality set $\fch$ of Chen fractions, let $\Q \fch$ denote the subspace of  functions in variables  $\{x_i\}_{i\in \Z _{\ge 1}}$ spanned  by  $\fch$ over $\Q$.
By \mcite{GPZ1}, The set $\fch$ is a linear basis of $\Q \fch$.

The linear space $\Q F ^{ch}$ is not an algebra under the function multiplication,  which  we will denote by  $``\cdot"$ if it needs to be shown explicitly.
For example \mcite {GPZ2}, the product of Chen fractions
$$\frac 1{(x_1+x_2)x_1}\cdot\frac 1{(x_1+x_3)x_1}$$
is not in $\Q \fch$.

This lack of multiplication closure is remedied in the locality setting.
Extend the locality relation $\top$ on $\fch$ bilinearly to yield a locality vector space $\Q \fch$:
$$\Big(\sum_i a_i f_i\Big) \top \Big(\sum_j b_j g_j\Big) \text{ if and only if } f_i \top g_j \text{ for all } i, j, \quad a_i, b_j\in \Q, f_i, g_j\in \fch.$$

\begin {prop}[ \mcite{GPZ1}]
\mlabel {prop:localchen}
The restriction of function multiplication  makes $(\Q \fch, \top_{\Q \fch}, 1)$ into a  locality commutative algebra.
\end{prop}

\begin{exam}
For $\wvec{1}{x}, \wvec{1}{y}\in \fch$,
$$
\wvec{1}{x}\cdot \wvec{1}{y}=\frac{1}{x}\cdot\frac{1}{y}=\frac{1}{(x+y)x}+\frac{1}{(x+y)y}=\wvec{1,1}{y, x}+\wvec{1,1}{x, y},
$$
which is in $\Q\fch$.
\end{exam}

A linear map $d:(A,\top)\to (A,\top)$ on a locality algebra is called a {\bf locality derivation} if for $(x,y)\in A\times_\top A$, we have
\begin{equation}
\begin{split}
	(d(x),y), (x,d(y))\in A\times_\top A, \\
	 d(xy)=d(x)y+xd(y).
	\end{split}
 \mlabel{eq:diffa}
\end{equation}
Then $(A,\top,d)$ is called a {\bf locality differential algebra}. If $A$ carried a family of commuting locality derivations, then $A$ is called a {\bf locality multi-differential algebra}.

As functions, the space $\Q \fch$ carries a family of partial derivatives $\frac {\partial }{\partial x_i}$, $i\in \Z _{\ge 1}$. Let
\begin{equation}
	\partial _i:=-\frac {\partial }{\partial x_i}
	\mlabel{eq:partial}
\end{equation}
Then $\Q\fch$ is a locality multi-differential commutative algebra.

The natural grading on $\fch$ given by the degree of a Chen fraction equips the locality algebra $A:=\Q \fch$ with a {\bf locality graded algebra} in the sense that, there is a grading $A=\oplus_{n\geq 0}A_n$ such that $m_A((A_m\times A_n)\cap \top_A) \subseteq A_{m+n}$ for all $m, n\in \Z _{\ge 0}$.
Furthermore, the derivation $\partial_i$ increases the grading by one.

We summarize the locality structures on the space of Chen fractions as follows.
\begin{prop}
\mlabel{p:chenlocda}
The quintuple $(\Q\fch,\top, \cdot, 1, \{\partial_i\}_i)$ is a locality multi-differential graded algebra.
\end{prop}

For later use, we also recall the following notion.
Given two locality algebras $(A_i, \top_i), i=1,2$,
a {\bf locality algebra homomorphism} is a linear map   $\varphi:A_1\longrightarrow A_2$  such that
\begin{equation}
a  \top_1 b\Longrightarrow \varphi(a)\top_2\varphi(b),\quad
\varphi(a \cdot b)=\varphi(a )\cdot \varphi(b), \quad \varphi(1_{A_1})=1_{A_2}.
\mlabel{eq:lochom}
\end{equation}

As noted in the introduction, the locality multi-differential algebra obtained in Proposition~\mref{p:chenlocda} provides a natural locality setting for the recent study of multi-differential algebra and multi-Novikov algebra arising from the regularity structures in stochastic PDEs~\mcite{BD,BHZ}.

\subsection{Locality multi-differential Hopf algebras}
\mlabel{ss:lochopf}
Here we give the notion of locality multi-differential Hopf algebras, in order to be applied to the locality multi-differential algebra of Chen fractions.

The coalgebra structure can be generalized to the locality setting~\mcite{CGPZ}.
For a locality vector space $(C, \top)$ over $\Q$, let
$C\otimes_\top  C$
denote the image of the composition map
$$  \Q   (C\times_\top C) \rightarrow \Q(C\times C)\rightarrow C\ot C,$$ where the first map is the inclusion and the second map is the quotient map modulo bilinearity.
\begin{enumerate}
  \item Let $\Delta:C\to C\otimes C$ be a linear map. The triple $(C,\top, \Delta)$ is called a {\bf  (counitary) locality  coalgebra} if  it satisfies the following conditions.
\begin{enumerate}
\item For $U\subset C$,
\begin{equation}
\Delta (U^\top) \subseteq U^\top \otimes _\top U^\top,
\mlabel{eq:comag}
\end{equation}
which is equivalent  to the condition that, for $c\in C$,
$$\Delta(c^{\top})\subseteq c^{\top}\otimes_{\top}  c^{\top}.
$$
\item  The coassociativity holds:
$$(\rmid _C\ot \Delta)\,\Delta= (\Delta\ot \rmid _C)\, \Delta.$$
\mlabel{localDelta}
\item
There is a {\bf counit}, namely a linear map $\vep: C\to   K  $ such that
  $ ({\rmid}_C\otimes   \vep)\, \Delta=  (\vep\otimes {\rmid}_C)\, \Delta= {\rmid}_C$.
\end{enumerate}
\item A {\bf graded locality coalgebra} is a locality coalgebra
 $(C,\top, \Delta)$ with a grading $C=\oplus_{n\geq 0} C_n$ such that, for $U\subseteq C$,
\begin{equation}
\Delta(C_n\cap U^\top)\subseteq\bigoplus_{p+q=n} (C_p\cap U^\top)\otimes_\top (C_q\cap U^\top).
\mlabel{eq:lconn}
\end{equation}
Moreover,  the  graded locality coalgebra   is called  {\bf connected}  if
$$
\bigoplus_{n\geq 1} C_n = \ker \vep,
$$
and so $C_0$ is one-dimensional.
\item
A {\bf locality bialgebra} is a sextuple $( B , \top, m, u, \Delta, \vep)$ consisting of a locality  algebra $( B , \top, m, u)$ and a locality coalgebra $\left( B , \top, \Delta, \vep\right)$
such that $\Delta$ and $\vep$ are locality algebra homomorphisms as stipulated in Eq.~\meqref{eq:lochom}.
\item
A  locality bialgebra $B$ is called  {\bf graded} if there is a $\Z_{\geq 0}$-grading $B=\oplus_{n\geq 0} B_n$ with respect to which $B$ is both a  graded locality  algebra and a graded  locality  coalgebra. In this case,  $B$  is called  {\bf connected}  if it is connected as a coalgebra.
\item
A {\bf locality Hopf algebra} is  a  septuple  $( B,\top, m, \Delta, u, \vep, S)$  such  that  $( B,\top, m, \Delta, u, \vep)$   is  a  locality   bialgebra,  $S:  B \to  B $ is  called  an  {\bf antipode}   which is   a linear map such that $a\top b$ implies $a\top S(b)$ for $a, b\in B$,
 and
 \[S\star \rmid _B=\rmid _B\star S = u  \, \vep.\]
\end{enumerate}

Generalizing the notion of differential (associative) Hopf algebra~
\mcite{Ar,Br}, we give the locality variation.
\begin{defn}
\mlabel{de:ldha}
A {\bf locality multi-differential Hopf algebra} is a locality Hopf algebra
$$( B,\top, m, \Delta, u, \vep, S)$$
together with a family $d_i:B\to B, i\geq 1,$ of commuting linear maps which are both {\bf locality derivations} as defined in Eq.~\meqref{eq:diffa} and {\bf locality coderivations} in the sense that, for $a\in B$,
\begin{equation}
	\mlabel{eq:codiffa}
	\Delta d_i (a^\top)= (\id\ot d_i+d_i\ot \id)\Delta (a^\top)\subseteq a^\top \ot_\top a^{\top}.
\end{equation}
\end{defn}

\subsection {Locality multi-differential Hopf algebra of  Chen  fractions}
We now apply the general notion of locality multi-differential Hopf algebras to the locality multi-differential algebra $(\Q\fch,\top, \partial_i)$ in Proposition~\mref{p:chenlocda}.

With the derivations $\partial _i:=-\frac {\partial }{\partial x_i}$, we define
\begin{equation}
 d_{i,j}:=\p_i-\p_j,
\mlabel{eq:partial2}
\end{equation}
which are again derivation. Then we have \mcite {GPZ2}
\begin{align}
 \wvec{s_1, s_2,\cdots,s_k }{x_{i_1},x_{i_2},\cdots,x_{i_k}}
=\frac{d_{i_1,i_0}^{s_1-1}  d_{i_2,i_1}^{s_2-1}\cdots   d_{i_k, i_{k-1}}^{s_k-1}}{(s_1-1)!(s_2-1)!\cdots (s_k-1)!}\wvec{1,1,\cdots,1}{x_{i_1}, x_{i_2},\cdots,x_{i_k}}. \mlabel{eq:chenfractionpartial}
\end{align}

\begin{defn}
Define a linear map
$$\fraccop : \Q \fch\longrightarrow \Q \fch\otimes \Q \fch$$
by the following recursion.
 \begin{enumerate}
\item  For the constant function $1$, define
$$\fraccop(1):=1\otimes 1.
$$
\item  For $s_1=\cdots=s_k=1$, define
\begin{equation}\mlabel{eq:deconcan}
\fraccop\big(\wvec{1,\cdots,1}{x_{i_1}, \cdots,x_{i_k}}\big):=\sum_{j=0}^{k}\wvec{1,\cdots ,1}{x_{i_1},\cdots, x_{i_j}}\otimes\wvec{1,\cdots ,1}{x_{i_{j+1}},\cdots, x_{i_k}}.
\end{equation}
\item Recursively, define
\begin{equation}
       \fraccop(\wvec{s_1,\cdots, s_{j}+1,\cdots, s_k}{x_{i_1},\cdots, x_{i_j},\cdots, x_{i_k}}):=\frac{\id\otimes d_{i_j, i_{j-1}}+d_{i_j, i_{j-1}}\otimes \id}{s_j}\fraccop(\wvec{s_1,\cdots ,s_{j},\cdots ,s_k}{x_{i_1},\cdots, x_{i_j},\cdots ,x_{i_k}}).
       \mlabel{eq:chenrec}
     \end{equation}
\end{enumerate}
\mlabel{d:chencop}
\end{defn}

Since the partial derivatives $\partial_i$ and hence all $d_{i,j}$ commute among themselves, iterations of the recursion does not depend on the order of applications. Therefore, the linear map $\fraccop $ is well-defined. Complete details of the proof mirrors that of Proposition~\mref{p:uniquecoproduct}.

Note that Eq.~\meqref{eq:chenrec} means
$$ \fraccop d_{i_j,i_{j-1}} = (\id\otimes d_{i_j, i_{j-1}}+d_{i_j, i_{j-1}}\otimes \id) \fraccop.$$
This is the defining identity that the derivations $d_{i,j}$ are also locality coderivations as in Eq.~\meqref{eq:codiffa}.

Define a linear  map $\varepsilon: \Q\fch\longrightarrow\Q$  on the basis $\fch$ by
$$\varepsilon(f)=\left\{\begin{array}{lc}
1, &{\rm if}\, f=1,\\
0, &{\rm otherwise}.
\end{array}
\right.
$$
\begin{prop}
	\mlabel{prop:Deltacoasso}
\mlabel{lem:localitycounitalcoalg}
The quadruple $(\Q \fch, \top, \fraccop, \varepsilon, \{\partial_i\}_i)$  is a locality multi-differential coalgebra.
\end{prop}

\begin{proof}
Since
$$ \p_i=\sum_{j=1}^i d_{j,j-1},,  \quad i\in \Z _{\ge 1}, $$
the recursion in Eq. \meqref{eq:chenrec} is equivalent to the coderivation condition
\begin {equation}
\mlabel {eq:DeltaDer}
\fraccop   \p _i=(\id\otimes \p _i+\p _i\otimes \id)  \fraccop, \quad i\in \Z_{\geq 1}.
\end {equation}

It then follows that
\begin{equation}
\begin{split}
(\id\otimes \fraccop)  (\id\otimes \p _i+\p _i\otimes \id)
=(\id\otimes \id\otimes \p _i+\id\otimes \p _i\otimes \id+\p _i\otimes \id\otimes \id )  (\id\otimes \fraccop), \\
(\fraccop \otimes \id)  (\id\otimes \p _i+\p _i\otimes \id)
=(\id\otimes \id\otimes \p _i+\id\otimes \p _i\otimes \id+\p _i\otimes \id\otimes \id )  (\fraccop \otimes \id).
\end{split}
\mlabel{eq:Deltacommpartial}
\end{equation}

Now we prove the coassociativity
$$(\id\otimes\fraccop)\fraccop(\wvec{s_1,\cdots,s_k}{x_{i_1},\cdots,x_{i_k}})=(\fraccop\otimes \id)\fraccop(\wvec{s_1,\cdots,s_k}{x_{i_1},\cdots,x_{i_k}})
$$
of $\fraccop$ by induction on $|\vec s|:=s_1+\cdots +s_k\geq k$. When $|\vec s|=k$, we have $s_1=\cdots=s_k=1$, and the coproduct is just the deconcatenation and hence is coassociative:
\begin{align*}
(\id\otimes\fraccop)\fraccop\big(\wvec{1,1,\cdots,1}{x_{i_1}, x_{i_2},\cdots,x_{i_k}}\big)=&\sum_{0\le j\le \ell\le k}\wvec{1,\cdots,1}{x_{i_1},\cdots,x_{i_j}}\otimes \wvec{1,\cdots,1}{x_{i_{j+1}}, \cdots,x_{i_\ell}}\otimes \wvec{1,\cdots,1}{x_{i_{\ell+1}},\cdots,x_{i_{k}}}\\
=&(\fraccop\otimes \id)\fraccop\big(\wvec{1,1,\cdots,1}{x_{i_1}, x_{i_2},\cdots,x_{i_k}}\big)
\end{align*}
Assume that $\fraccop$ is coassociative for Chen fractions with $|\vec s|=\ell\ge k$. Then for a Chen fraction $\wvec{s_1,\cdots,s_k}{x_{i_1},\cdots,x_{i_k}}$ with $|\vec s|=\ell+1$, there is $s_j\geq 2$. So we have
\begin{align*}
 &(\id\otimes \fraccop)\fraccop(\wvec{s_1\cdots s_j\cdots s_k}{x_{i_1}\cdots x_{i_j}\cdots x_{i_k}})\\
=&(\id\otimes \fraccop)\Big(\frac{\id\otimes d_{{i_j},{i_{j-1}}}+d_{{i_j},{i_{j-1}}}\otimes \id}{s_j-1}\fraccop(\wvec{s_1\cdots s_{j}-1\cdots s_k}{x_{i_1}\cdots x_{i_j}\cdots x_{i_k}})\Big)\\
\overset {\textcircled{1}}{=}&\frac{\id\otimes \id\otimes d_{{i_j},{i_{j-1}}}+\id\otimes d_{{i_j},{i_{j-1}}}\otimes \id+d_{{i_j},{i_{j-1}}}\otimes \id\otimes \id}{s_j-1}(\id\otimes \fraccop)\fraccop(\wvec{s_1\cdots s_{j}-1\cdots s_k}{x_{i_1}\cdots x_{i_j}\cdots x_{i_k}})\\
\overset {\textcircled{2}}{=}&\frac{\id\otimes \id\otimes d_{{i_j},{i_{j-1}}}+\id\otimes d_{{i_j},{i_{j-1}}}\otimes \id+d_{{i_j},{i_{j-1}}}\otimes \id\otimes \id}{s_j-1}(\fraccop \otimes \id)\fraccop(\wvec{s_1\cdots s_{j}-1\cdots s_k}{x_{i_1}\cdots x_{i_j}\cdots x_{i_k}})\\
\overset {\textcircled{3}}{=}&(\fraccop\otimes \id)\fraccop(\wvec{s_1\cdots s_{j}\cdots s_k}{x_{i_1}\cdots x_{i_j}\cdots x_{i_k}}),
\end{align*}
where $\textcircled{1}$ follows from Eq.~(\mref{eq:Deltacommpartial}), $\textcircled{3}$ follows from Eqs.~(\mref{eq:Deltacommpartial}) and (\mref{eq:chenrec}), and $\textcircled{2}$ follows from the induction hypothesis. This completes the inductive proof that $\fraccop$ is coassociative.

We further show that $\fraccop$ is a locality coproduct, namely, $\fraccop(\{f\}^{\top})\subseteq \{f\}^{\top}\otimes_{\top} \{f\}^{\top}$ for arbitrary $f\in \fch.$
To begin with, it is known that $1\in \{f\}^{\top}$,   and
	$$\fraccop(1)=1\otimes 1\in \{f\}^{\top}\otimes _{\top} \{f\}^{\top}.$$
Next let $g=\wvec{s_1,s_2,\cdots,s_k}{x_{i_1}, x_{i_2},\cdots,x_{i_k}}$ be in $\{f\}^{\top}$. This means that $g$ and $f$ have disjoint variables. Since $\fraccop(g)$ and $g$ depend on the same variables, we have $\fraccop(g)\in \{f\}^{\top}\otimes_\top \{f\}^{\top}$. Thus
	$$\fraccop(\wvec{s_1,s_2,\cdots,s_k}{x_{i_1}, x_{i_2},\cdots,x_{i_k}})\in \{f\}^{\top}\otimes_\top \{f\}^{\top},
	$$
as needed.
	
Finally, we show that $\varepsilon$ is a counit. This follows from
$$(\varepsilon\otimes {\id})\fraccop(h)=\varepsilon(1) h=h,
\quad ({\id}\otimes \varepsilon)\fraccop(h)=\varepsilon(1) h=h, \quad
h\in \Q\fch.
$$

In summary, we have proved that  $(\Q\fch, \top, \fraccop, \varepsilon)$  is  a locality  coalgebra, and together with Eq.~\meqref{eq:DeltaDer}, a locality multi-differential coalgebra.
\end{proof}
We now show the compatibility of the coproduct with the product.
\begin{prop}
\mlabel{lem:Deltahomomorphism}
For $f,g\in\Q \fch$, if $f\top g$, then
\begin{equation}
\fraccop(f\cdot g)=\fraccop(f)\cdot\fraccop(g).
\mlabel{eq:delhom}
\end{equation}
\end{prop}
\begin{proof}
We only need to prove this for $f,g\in \fch$.  The  conclusion is obvious if $f=1$ or $g=1$.

Let $f=\wvec{s_1,\cdots, s_k}{x_{i_1},\cdots ,x_{i_k}}$ and $g=\wvec{t_1,\cdots, t_\ell}{x_{j_1},\cdots , x_{j_\ell}}$. Since $f\top g$, the variables $x_{i_1},\cdots, x_{i_k}$, $x_{j_1},\cdots, x_{j_\ell}$  are distinct.

We apply the induction on the sum of weights 
$$w:=|\vec s|+|\vec t|=s_1+\cdots+s_k+t_1+\cdots+t_\ell\geq k+\ell.$$
If $w=k+\ell$, then $s_1=\cdots=s_k=t_1=\cdots=t_\ell=1$. In this case, the product of $f$ and $g$ follows the shuffle product and Eq.~\meqref{eq:delhom} is the known compatibility of the shuffle product and  the deconcatenation  coproduct in the usual shuffle Hopf algebra \mcite{H2}.

Now let $w=m\geq k+\ell$ and assume that Eq.~\meqref{eq:delhom} holds for  $w=m$. Consider the case of $w=m+1$. So either $s_p\geq 2$ for some $1\leq p\leq k$ or $t_q\geq 2$  for  some $1\leq q\leq \ell$. We only consider the former case, since the proof is the same for the latter case. Then we have
\begin{align*}
 &\fraccop\Big(\wvec{s_1,\cdots, s_k}{x_{i_1},\cdots, x_{i_k}}\cdot\wvec{t_1,\cdots, t_\ell}{x_{j_1},\cdots ,x_{j_\ell}}\Big)=\fraccop\Big(\frac{d_{i_p, i_{p-1}}}{s_p-1}(\wvec{s_1,\cdots,s_p-1,\cdots, s_k}{x_{i_1},\cdots,x_{i_p},\cdots, x_{i_k}})\cdot\wvec{t_1,\cdots, t_\ell}{x_{j_1},\cdots ,x_{j_\ell}}\Big)\\
\overset {\textcircled{1}}{=}&\fraccop\Big(\frac{d_{i_p, i_{p-1}}}{s_p-1}\big(\wvec{s_1,\cdots,s_p-1,\cdots, s_k}{x_{i_1},\cdots,x_{i_p},\cdots, x_{i_k}}\cdot\wvec{t_1,\cdots, t_\ell}{x_{j_1},\cdots ,x_{j_\ell}}\big)\Big)\\
=&\frac{\id\otimes d_{i_p, i_{p-1}}+d_{i_p, i_{p-1}}\otimes \id}{s_p-1}\Big(\fraccop\big(\wvec{s_1,\cdots,s_p-1,\cdots, s_k}{x_{i_1},\cdots,x_{i_p},\cdots, x_{i_k}}\cdot\wvec{t_1,\cdots, t_\ell}{x_{j_1},\cdots ,x_{j_\ell}}\big)\Big)\\
\overset {\textcircled{2}}{=}&\frac{\id\otimes d_{i_p, i_{p-1}}+d_{i_p, i_{p-1}}\otimes \id}{s_p-1}\Big(\fraccop(\wvec{s_1,\cdots, s_p-1,\cdots ,s_k}{x_{i_1},\cdots,x_{i_p},\cdots, x_{i_k}})\cdot\fraccop(\wvec{t_1,\cdots , t_\ell}{x_{j_1},\cdots, x_{j_\ell}})\Big)\\
\overset {\textcircled{3}}{=}&\frac{\id\otimes d_{i_p, i_{p-1}}+d_{i_p, i_{p-1}}\otimes \id}{s_p-1}\Big(\fraccop(\wvec{s_1,\cdots, s_p-1,\cdots ,s_k}{x_{i_1},\cdots,x_{i_p},\cdots, x_{i_k}})\Big)\cdot\fraccop(\wvec{t_1,\cdots , t_\ell}{x_{j_1},\cdots, x_{j_\ell}})\\
=&\fraccop\Big(\frac{d_{i_p, i_{p-1}}}{s_p-1}(\wvec{s_1,\cdots,s_p-1,\cdots, s_k}{x_{i_1},\cdots,x_{i_p},\cdots, x_{i_k}})\Big)\cdot\fraccop(\wvec{t_1,\cdots, t_\ell}{x_{j_1},\cdots ,x_{j_\ell}})\\
=&\fraccop(\wvec{s_1,\cdots ,s_p,\cdots, s_k}{x_{i_1},\cdots,x_{i_p},\cdots, x_{i_k}})\cdot\fraccop(\wvec{t_1,\cdots ,t_\ell}{x_{j_1},\cdots, x_{j_\ell}}).
\end{align*}
Here the equalities labeled by $\textcircled{1}$ and $\textcircled{3}$ follow from the fact that the variables in the first part are disjoint from the variables in the second part by the locality condition. $\textcircled{2}$ followed by induction hypothesis.

This completes the induction.
\end{proof}

\begin{theorem}
The  sextuple  $(\Q \fch, \top ,  \cdot, \fraccop,  1, \varepsilon, \{\partial_i\}_i)$ is  a locality connected multi-differential bialgebra and hence a locality multi-differential Hopf algebra as defined in Definition~\mref{de:ldha}.
\mlabel{t:fchDeltalocalhopf}
\end{theorem}

\begin{proof} By Proposition \mref {p:chenlocda}, $(\Q \fch, \top, \cdot, 1,\{\partial_i\}_i)$ is  a  locality multi-differential algebra. By Proposition  \mref{lem:localitycounitalcoalg},  $(\Q\fch,$ $\top, \fraccop, \varepsilon,\{\partial_i\}_i)$  is  a  locality multi-differential coalgebra. Also, by  Lemma  \mref{lem:Deltahomomorphism},  $\fraccop$   is  a  locality  algebra  homomorphism  with  respect  to  the locality multiplication  $``\cdot"$.
	
We now check that $\varepsilon$ is a locality algebra homomorphism as follows.
\begin{enumerate}
  \item $\varepsilon(1)=1=\varepsilon (1)\cdot \varepsilon(1)$,
  \item For basis elements $f, g\in \fch$  such that  $f\top \ g$ and  $f\neq 1$ or  $g\neq 1$,
   $$\varepsilon(f\cdot g)=0=\varepsilon(f)\cdot\varepsilon(g).$$
\end{enumerate}
So  $(\Q \fch, \top ,  \cdot, \fraccop, 1, \varepsilon)$ is  a  locality  bialgebra.

We next show that the locality bialgebra $(\Q \fch, \top,  \cdot, \fraccop, 1, \varepsilon)$ is a connected graded locality bialgebra.
Let
$$F_0:=\{1\},
\quad F_m:=\big\{\wvec{s_1,\cdots,s_k}{x_{i_1},\cdots,x_{i_k}}\,\big| s_1+\cdots+s_k=m\,\big\}.
$$
Then we have the disjoint union
$\fch=\bigsqcup_{m=0}^{\infty}F_m,
$
and  $\Q\fch$  is a graded vector space:
$$\Q\fch=\bigoplus_{m=0}^{\infty}\Q F_m.$$
Since the grading is given by the degrees of fractions,   $(\Q \fch, \top,  \cdot, 1)$  is  a graded locality algebra.

To show that $(\Q\fch, \top, \fraccop, \varepsilon)$  is  a graded locality   coalgebra  under  this   grading,
fix some $U\subseteq \Q \fch$. Then  $1\in U^{T}$ and we  have
 $$\fraccop(1)=1\otimes 1\in (\Q F_0\cap U^{\top})\otimes_{\top}(\Q F_0\cap U^{\top}).
 $$
Next we verify
\begin{equation}
\fraccop(\Q F_m\cap U^{\top})\subseteq \bigoplus_{p+q=m} (\Q F_p\cap U^{\top})\otimes_{\top} (\Q F_q \cap U^{\top})
\mlabel{eq:cograde}
\end{equation}
by  induction  on  $m\geq 0$.
Let $\sum_{i}a_i h_i\in U^{\top}$  with $h_i\in \fch$. By the definition  of  locality  relation in $\Q \fch$, we have $h_i\in U^{\top}$. So  we  only  need to prove Eq. \meqref{eq:cograde} for basis elements, which take the form $h=\wvec{s_1,\cdots, s_k}{x_{i_1},\cdots, x_{i_k}}\in F_{m}\cap U^{\top}$.
Now the induction on $m\geq k$ proceeds as follows.

For the initial step $m=k$, we have $s_1=\cdots=s_k=1$ and  $h=\wvec{1,\cdots,1}{x_{i_1},\cdots,x_{i_k}}\in U^{\top}$. Then
 $$\fraccop(\wvec{1,\cdots,1}{x_{i_1},\cdots,x_{i_{k}}})=\sum_{j=0}^{k}\wvec{1,\cdots ,1}{x_{i_1},\cdots, x_{i_j}}\otimes\wvec{1,\cdots ,1}{x_{i_{j+1}},\cdots, x_{i_{k}}},
 $$
 which belongs to  $\bigoplus_{p+q=k} (\Q F_p\cap U^{\top})\otimes_{\top} (\Q F_q\cap U^{\top})$.

  Assume that Eq.~\meqref{eq:cograde} holds for $m\geq k$ and consider $\wvec{s_1,\cdots, s_k}{x_{i_1},\cdots, x_{i_k}}\in F_{m+1}\cap U^{\top}$. Then there is $s_j\ge 2$  for  some $1\le j\le k$. By Eq.(\ref{eq:chenrec}), we have
 {\small
 \begin{equation}
 \fraccop(\wvec{s_1,\cdots, s_{j},\cdots, s_k}{x_{i_1},\cdots, x_{i_j},\cdots, x_{i_k}}):=\frac{\id\otimes d_{i_j, i_{j-1}}+d_{i_j, i_{j-1}}\otimes \id}{s_j-1}\fraccop(\wvec{s_1,\cdots ,s_{j}-1,\cdots ,s_k}{x_{i_1},\cdots, x_{i_j},\cdots ,x_{i_k}}).
\mlabel{eq:coph}
\end{equation}}
The fraction  $\wvec{s_1,\cdots ,s_{j}-1,\cdots ,s_k}{x_{i_1},\cdots, x_{i_j},\cdots ,x_{i_k}}$ is in $F_m\cap U^{\top}$. So the induction hypothesis gives
 $$\fraccop(\wvec{s_1,\cdots ,s_{j}-1,\cdots ,s_k}{x_{i_1},\cdots, x_{i_j},\cdots ,x_{i_k}})\in \bigoplus_{p+q=m} (\Q F_p\cap U^{\top})\otimes_{\top} (\Q F_q\cap U^{\top}).
 $$
Thus applying Eq.~\meqref{eq:coph} yields
 $$\fraccop(\wvec{s_1,\cdots, s_{j},\cdots, s_k}{x_{i_1},\cdots, x_{i_j},\cdots, x_{i_k}})\in \bigoplus_{p+q=m+1} (\Q F_p \cap U^{\top})\otimes_{\top} (\Q F_q \cap U^{\top}).
 $$
To summarize, the sextuple  $(\Q\fch, \top, \cdot, \fraccop, 1, \varepsilon )$ is a locality graded bialgebra.

Obviously,
$$
\ker \vep=\bigoplus_{m\geq 1} \Q F_m, \quad \Q F_0= \Q\, 1.
$$
Hence $(\Q\fch, \top, \cdot, \fraccop, 1, \varepsilon )$ is a connected locality bialgebra, and thus a  locality Hopf algebra by~\cite[Proposition 5.6]{CGPZ}.
\end{proof}

\section {The MZV shuffle Hopf algebra}
\mlabel{s:mzvhopf}

Now we use a locality algebra homomorphism
$$\pi: \Q \fch\longrightarrow \calhd $$
to transfer the locality Hopf algebra structure on the space of Chen fractions $\Q\fch$ to a Hopf algebra structure on $\calhd$, which is then identified with the one introduced in Theorem~\mref{thm:shuhopf} defined by an operator recursion, thereby proving Theorem~\mref{thm:shuhopf}.

\subsection {The passage from Chen fractions to the \emzvsha} The locality multiplication of Chen fractions as functions is closely related to the shuffle product of MZVs.

Let $\cala =\{x_i\}_{i\in \Z _{\ge 0}} $ be an alphabet. As we have already seen, the linear space $\Q W$ with a basis $W=W_\cala$ of words in $\cala$ carries the shuffle product $\shap$. Let $W_1$ be the subset of words of the form
$$x_0^{s_1-1}x_{i_1}\cdots x_0^{s_k-1}x_{i_k},
$$
with $s_1, \cdots, s_k \in \Z _{\ge 1}$ and $i_1, \cdots , i_k$ distinct. Let ${\bf1}$ be the empty word. The set $W_1$ has a natural locality relation $\top_{W_1}$
\begin{equation}
\left\{\begin{array}{ll} (x_0^{s_1-1}x_{i_1}\cdots x_0^{s_k-1}x_{i_k})\top_{W_1} (x_0^{t_1-1}x_{j_1}\cdots x_0^{t_\ell -1}x_{j_\ell}), & \text{if }
\{i_1, \cdots, i_k\}\cap \{j_1, \cdots, j_\ell \}=\emptyset, \\
{\bf 1} \top_{W_1}w,&  w\in W_1.
\end{array} \right .
\mlabel{eq:wordloc}
\end{equation}
This locality relation extends to a locality relation $\top _{\Q W_1}$ on $\Q W_1$. It is easy to see that the restriction of the shuffle product on $\Q W$ equips $\Q W_1$ with a locality commutative algebra structure. We further have
\begin {prop}
\mlabel{prop:psimorphism}
There is the algebra homomorphism
\begin{equation}
	\psi: (\Q W,\shap)\to (\Q \langle x_0, x_1 \rangle x_1,\shap),
\quad \left \{ \begin{array}{l} {\bf1}\longmapsto {\bf1}, \\
x_0^{s_1-1}x_{i_1}\cdots x_0^{s_k-1}x_{i_k}\mapsto  x_0^{s_1-1}x_{1}\cdots x_0^{s_k-1}x_{1}.
\end{array} \right .
\mlabel{eq:psi}
\end{equation}
With the locality on $\Q W_1$ defined in Eq.~\meqref{eq:wordloc} and the full locality $\top_{\Q \langle x_0, x_1 \rangle x_1} =\Q \langle x_0, x_1 \rangle x_1\times \Q \langle x_0, x_1 \rangle x_1$ on $\Q \langle x_0, x_1 \rangle x_1$, $\psi$ restricts to a  locality algebra homomorphism
$$\psi: (\Q W_1,\shap)\to (\Q \langle x_0, x_1 \rangle x_1,\shap).$$
\end{prop}

On the other hand, we quote the following result.
\begin {prop} \mcite {GPZ1}
\mlabel{prop:phiisomorphism}
The the linear map
\begin{equation}
	\phi: \Q \fch \to \Q W_1,\quad \left\{\begin{array}{l} 1\longmapsto {\bf1},\\
		\wvec{s_1, \cdots, s_k}{x_{i_1},\cdots, x_{i_k}} \mapsto x_0^{s_1-1}x_{i_1}\cdots x_0^{s_k-1}x_{i_k},
	\end{array}\right .
\mlabel{eq:phi}
\end{equation}
is a locality algebra isomorphism from $(\Q F ^{ch}, \top, \cdot)$ to
$(\Q W_1, \top_{\Q W_1}, \shap)$, where $\cdot$ is the function multiplication.
\end{prop}

\begin {remark} Proposition \mref{prop:phiisomorphism} demonstrates the advantage of working in the locality setting. The linear map $\phi$ is not an algebra homomorphism without the locality condition. For example,
$$\phi (\wvec{1}{x_{1}}\wvec{1}{x_{1}})=\phi (\frac 1{x_1^2})=\phi (\wvec{2}{x_{1}})=x_0x_1;
$$
but
$$
\phi (\wvec{1}{x_{1}})\shap \phi (\wvec{1}{x_{1}})=x_1\shap x_1=2x_1x_1.
$$
\end{remark}

Combining Propositions \mref {prop:psimorphism}, \mref {prop:phiisomorphism}, and Eq. (\mref {eq:rho}), we have
\begin {coro}
\mlabel {coro:pi} The composition
$$\pi =\rho ^{-1}  \psi   \phi : (\Q \fch,\cdot) \to (\calhd,\shapt)$$
as shown in the diagram
\begin{equation}
\begin{split}
\xymatrix{ \Q W_1 \ar[rr]^\psi && \Q \langle x_0, x_1 \rangle x_1 \\
\Q \fch \ar@{-->}[rr]^\pi \ar[u]^\phi && \calhd \ar[u]^\rho}
\end{split}
\mlabel{eq:phipsidiag}
\end{equation}
is a locality algebra homomorphism, with the locality function multiplication on $\Q \fch$ and the locality shuffle product $\shapt$ on $\calhd$, equipped with the full locality $\calhd\times \calhd$.
\end{coro}
\begin{remark}\mlabel{rem:piotimespiishomo}
Consequently, $\pi\otimes\pi: \Q \fch \otimes\Q\fch \to \calhd\otimes\calhd$ is also a locality algebra homomorphism.
\end{remark}
\subsection{The descent Hopf algebra} For $[s_1,\cdots,s_k]\in \calhd$, take subsets $\{i_1,\cdots, i_k\}$ and $\{j_1,\cdots, j_k\}$ of  $k$ distinct elements in $\Z_{\ge 1}$. By the definition of $\fraccop: \Q \fch \to \Q \fch\otimes \Q \fch$, we see that $\fraccop$  commutes  with  changing of  variables:
 $$\Big(\fraccop\big(\wvec{s_1,\cdots,s_k}{x_{i_1},\cdots,x_{i_k}}\big)\Big)\Big|_{x_{i_1}\mapsto x_{j_1}, \cdots, x_{i_k}\mapsto x_{j_k}} =\fraccop\big(\wvec{s_1,\cdots,s_k}{x_{j_1},\cdots,x_{j_k}}\big).
 $$
Thus we have
 \begin{lemma}
 \mlabel{lem:nodependchoice}
 The element  $(\pi\otimes \pi)\fraccop(\wvec{s_1,\cdots,s_k}{x_{i_1},\cdots,x_{i_k}})$  does  not depend  on  the choice of the distinct elements $i_1,\cdots,i_k\in \Z_{\geq 1}$.
 \end{lemma}

 Therefore we obtain a linear map

\begin{equation}
\descop: \calhd\longrightarrow \calhd\otimes \calhd,
 \left\{ \begin{array}{l} \descop({\bf1}):={\bf1}\otimes {\bf1}, \\ \descop([s_1,\cdots, s_k]):=(\pi\otimes \pi)\fraccop(\wvec{s_1,\cdots,s_k}{x_{i_1},\cdots,x_{i_k}}),
 \{i_1,\cdots, i_k\}\subseteq \Z_{\ge 1},
\end{array} \right .
\mlabel{eq:descop}
\end{equation}
which is well-defined by Lemma \ref{lem:nodependchoice}.
The definition of $\descop$ gives
\begin{equation}
	\descop  \pi=(\pi\otimes \pi)  \fraccop.
\mlabel{eq:pidelta}
\end{equation}
Corollary~\mref{coro:pi} (resp. Eq.~\meqref{eq:pidelta}) amounts to say that the left square (resp. right square) of following diagram commutes.
$$\xymatrix{
\Q\fch \ot \Q\fch \ar[d]_{\pi \ot \pi} \ar[rr]^{ \cdot } &&  \Q\fch \ar[d]_{\pi} \ar[rr]^{\fraccop} && \Q\fch\otimes\Q\fch \ar[d]^{\pi\otimes\pi} \\
\calhd\ot \calhd \ar[rr]^{\shapt} &&  \calhd \ar[rr]^{\descop} && \calhd\otimes \calhd   }
$$

Thus we obtain
\begin{lemma}
For any  set  with $k+\ell$ distinct positive integers
$$\{i_1,\cdots,i_k, j_1,\cdots, j_{\ell}\},$$
 and $(s_1,\cdots,s_k)\in \Z_{\ge 1}^{k}, (t_1,\cdots, t_\ell)\in\Z_{\ge 1}^\ell$, we have
$$(\pi\otimes\pi)\fraccop(\wvec{s_1,\cdots,s_k}{x_{i_1},\cdots,x_{i_k}}\cdot\wvec{t_1,\cdots,t_\ell}{x_{j_1},\cdots,x_{j_\ell}})=\descop([s_1,\cdots,s_k]\shapt [t_1,\cdots,t_\ell]).
$$
\mlabel{lem:only 1}
\end{lemma}

\begin{theorem}
With $\vep_{\ge 1}$ defined in Eq.~\meqref{eq:counit}, the quintuple $(\calhd, \shap, {\bf 1}, \descop, \vep_{\ge 1})$  is a connected graded bialgebra and hence a Hopf algebra.
	\mlabel{t:deschopf}
\end{theorem}

\begin{proof}
First by Eq.~\meqref{eq:pidelta}, we have
$$	(\id\otimes\descop)\descop\pi =(\id\otimes \descop)(\pi\otimes \pi)\fraccop=(\pi\otimes\descop\pi)\fraccop
	=(\pi\otimes\pi\otimes\pi)(\id\otimes \fraccop)\fraccop,
$$	
and
$$ 	(\descop\otimes \id)\descop\pi =(\descop\otimes \id)(\pi\otimes \pi)\fraccop=(\descop\pi\otimes \pi)\fraccop
	=(\pi\otimes\pi\otimes\pi)(\fraccop\otimes \id)\fraccop.
$$
Thus the coassiciativity of $\descop$ follows from that of $\fraccop$ and the surjectivity of $\pi$.

We next  check that $\vep_{\ge 1}$ is an algebra  homomorphism  on basis elements.
For $\vec s\in \Z_{\ge 1}^k$, $\vec t\in \Z_{\ge 1}^{\ell}$, $k,\ell\ge 1$, we have
$$\vep_{\ge 1}({\bf1}\shapt {\bf 1})=1=\vep_{\ge 1}({\bf1})\vep_{\ge 1}({\bf1}),
$$
$$\vep_{\ge 1}({\bf1}\shapt [\vec s])=\vep_{\ge 1}([\vec s]\shapt {\bf 1})=\vep_{\ge 1}([\vec s])=0=\vep_{\ge 1}([\vec s])\vep_{\ge 1}({\bf 1}),
$$
$$\vep_{\ge 1}([\vec s]\shapt [\vec t])=0=\vep_{\ge 1}([\vec s])\vep_{\ge 1}([\vec t]).
$$

We then check that the coproduct  $\descop$ is an algebra homomorphism. For any $x\in \calhd$,
$$\descop(x\shapt {\bf 1})=\descop({\bf1}\shapt x)=\descop(x)=\descop({\bf1})\shapt \descop(x)=\descop(x)\shapt \descop({\bf1}).
$$
For basis elements $[s_1,\cdots, s_k]$ and $[t_1,\cdots, t_\ell]$, let $\{i_1,\cdots,i_k,j_1,\cdots,j_\ell\}$
be any $k+\ell$ distinct elements in $\Z_{\ge 1}$. Then by Lemma \mref{lem:only 1}, Lemma \mref{lem:Deltahomomorphism}, Corollary \mref {coro:pi} and  Remark \mref{rem:piotimespiishomo}, we obtain
\begin{align*}
 &\descop([s_1,\cdots,s_k]\shapt [t_1,\cdots,t_\ell])=(\pi\otimes\pi)\fraccop(\wvec{s_1,\cdots,s_k}{x_{i_1},\cdots,x_{i_k}}\cdot\wvec{t_1,\cdots,t_\ell}{x_{j_1},\cdots,x_{j_\ell}})\\
=&(\pi\otimes\pi)\big(\fraccop(\wvec{s_1,\cdots,s_k}{x_{i_1},\cdots,x_{i_k}})\cdot\fraccop(\wvec{t_1,\cdots,t_\ell}{x_{j_1},\cdots,x_{j_\ell}})\big)\\
=&\big((\pi\otimes\pi)\fraccop(\wvec{s_1,\cdots,s_k}{x_{i_1},\cdots,x_{i_k}})\big)\shapt \big((\pi\otimes \pi)\fraccop(\wvec{t_1,\cdots,t_\ell}{x_{j_1},\cdots,x_{j_\ell}})\big)\\
=&\descop([s_1,\cdots,s_k])\shapt\descop([t_1,\cdots,t_\ell]),
\end{align*}
as needed.

In summary, we have proved that the quintuple $(\calhd, \shapt, \descop, {\bf1}, \vep_{\ge 1})$  is a bialgebra.

We finally show that $(\calhd, \shapt, {\bf 1}, \descop, \vep_{\ge 1})$ is a connected graded bialgebra.
The space $\calhd$ has the grading by weight, that is,
$$\calhd=\bigoplus_{m=0}^{\infty}\Q H_m, \text{ with } \left\{\begin{array}{l} H_0:={\bf 1},\\
H_m:=\big\{[s_1,\cdots,s_k]\in \Z_{\ge 1}^k\,\big|\, s_1+\cdots+s_k=m, k\geq 1\big\}.
\end{array} \right . $$
By the definition  of  $\shapt$,  we have
$$ H_m\shapt H_n\subseteq \Q H_{m+n}, \quad m\ge 0, n\ge 0.$$
Moreover, since  $\pi: \Q\fch\longrightarrow \calhd$ is grading preserving,
we have
$$\descop(H_m)=(\pi\otimes \pi)\fraccop(F_m)\subseteq \bigoplus_{p+q=m} \pi(\Q F_p)\otimes\pi(\Q F_q)\subseteq \bigoplus_{p+q=m} \Q H_p\otimes\Q H_q.
$$

Therefore,  $(\calhd, \shapt, {\bf 1}, \descop, \vep_{\ge 1})$ is  a  connected  graded  bialgebra.
Consequently, as is well-known (see~\mcite{GG,Md} for example), $(\calhd, \shapt, {\bf 1}, \descop, \vep_{\ge 1})$  is a Hopf algebra.
\end{proof}

\subsection {Identifying the coproducts}
\mlabel{ss:twocoprod}
In this last part, we show that the descended coproduct $\descop$ satisfies the recursion in Definition~\mref{d:reccop} and thus is identified with the recursively defined coproduct $\reccop$ there.

\begin{lemma}
\mlabel{lem:partialanddelta}
\begin{enumerate}
  \item  For  the  constant  function  $1\in \fch$, we have
  $$\pi  \p_i(1)=\delta_i  \pi(1).$$
  \mlabel{i:pardel1}
  \item  For  $\wvec{s_1,\cdots, s_k}{x_{\ell_1}, \cdots, x_{\ell_k}}\in \fch$ and  $i\ge 1$, we have
 $$\pi  \p_{\ell_i}\big(\wvec{s_1,\cdots, s_k}{x_{\ell_1}, \cdots, x_{\ell_k}}\big)=\delta_i  \pi\big(\wvec{s_1,\cdots, s_k}{x_{\ell_1}, \cdots, x_{\ell_k}}\big).
$$
\mlabel{i:pardel2}
\end{enumerate}
\end{lemma}

\begin{proof} \eqref{i:pardel1} We directly check that $\pi  \p_i(1)=0$ and  $\delta_i  \pi(1)=\delta_i({\bf 1})=0$.

\eqref{i:pardel2} Let $\wvec{s_1,\cdots, s_k}{x_{\ell_1}, \cdots, x_{\ell_k}} \in \fch$.
\begin{enumerate}
  \item[(a)] If  $i\le k$, then
\begin{align*}
\pi  \p_{\ell_i}(\wvec{s_1,\cdots, s_k}{x_{\ell_1}, \cdots, x_{\ell_k}})=&\pi\Big(\sum_{j=1}^is_j\wvec{s_1,\cdots,s_j+1,\cdots,s_k}{x_{\ell_1},\cdots,x_{\ell_j},\cdots,x_{\ell_k}}\Big)\\
=&\sum_{j=1}^is_j[s_1,\cdots,s_j+1,\cdots,s_k]\\
=&\delta_i([s_1,\cdots, s_k])\\
=&\delta_i \pi(\wvec{s_1,\cdots, s_k}{x_{\ell_1}, \cdots, x_{\ell_k}}).
\end{align*}
\item[(b)] If $i>k$, then
$$ \hspace{4cm}
\pi  \p_{\ell_i}(\wvec{s_1,\cdots, s_k}{x_{\ell_1}, \cdots, x_{\ell_k}})=0
=\delta_i  \pi(\wvec{s_1,\cdots, s_k}{x_{\ell_1}, \cdots, x_{\ell_k}}).
\hspace{4cm}
\qedhere
$$
\end{enumerate}
\end{proof}

\begin{lemma}
 \mlabel{lem:tensorpartialanddelta}
Let  $B=1$ and $A=1$, or $A=\delta_j$ and $B=\p_{\ell_j}$. Then for Chen  fractions $\wvec{s_1,\cdots, s_k}{x_{\ell_1}, \cdots, x_{\ell_k}}$, $\wvec{t_1,\cdots,t_m}{x_{\ell_{k+1}}, \cdots, x_{\ell_{k+m}}}$ with distinct $\ell _1, \cdots, \ell_{k+m}$,  and $i,j\ge 1$, we have
\begin{align*} (\pi\otimes \pi)(B\otimes \p_{\ell_i})(1\otimes \wvec{t_1,\cdots,t_m}{x_{\ell_{k+1}}, \cdots, x_{\ell_{k+m}}})
  =&(A\cks \delta_i)(\pi\otimes \pi)(1\otimes \wvec{t_1,\cdots,t_m}{x_{\ell_{k+1}}, \cdots, x_{\ell_{k+m}}}),\\
(\pi\otimes \pi)(B\otimes \p_{\ell_i})(\wvec{s_1,\cdots, s_k}{x_{\ell_1}, \cdots, x_{\ell_k}}\otimes 1)
  =&(A\cks \delta_i)(\pi\otimes \pi)(\wvec{s_1,\cdots, s_k}{x_{\ell_1}, \cdots, x_{\ell_k}}\otimes 1), \\
(\pi\otimes \pi)(B\otimes \p_{\ell_i})(\wvec{s_1,\cdots, s_k}{x_{\ell_1}, \cdots, x_{\ell_k}}\otimes \wvec{t_1,\cdots,t_m}{x_{\ell_{k+1}}, \cdots, x_{\ell_{k+m}}})
  =&(A\cks \delta_i)(\pi\otimes \pi)(\wvec{s_1,\cdots, s_k}{x_{\ell_1}, \cdots, x_{\ell_k}}\otimes \wvec{t_1,\cdots,t_m}{x_{\ell_{k+1}}, \cdots, x_{\ell_{k+m}}}).
  \end{align*}
\end{lemma}
Here the arguments of the functions in each equality need to be specified in order to accurately define the actions of the partial derivatives.

\begin{proof} First let $B=1$ and $A=1$. Then by  Lemma \mref{lem:partialanddelta}, we have
\begin{align*}
  &(\pi\otimes \pi)(\id\otimes \p_{\ell_i})(1\otimes \wvec{t_1,\cdots,t_m}{x_{\ell_{k+1}}, \cdots, x_{\ell_{k+m}}})={\bf1}\otimes  \pi\p_{\ell_i}(\wvec{t_1,\cdots,t_m}{x_{\ell_{k+1}}, \cdots, x_{\ell_{k+m}}})\\
 =&{\bf1}\otimes  \delta_i\pi(\wvec{t_1,\cdots,t_m}{x_{\ell_{k+1}}, \cdots, x_{\ell_{k+m}}})\\
 =&(\id\cks \delta_i)(\pi\otimes \pi)(1\otimes \wvec{t_1,\cdots,t_m}{x_{\ell_{k+1}}, \cdots, x_{\ell_{k+m}}}).
  \end{align*}

Next let $B=\p_{\ell_j}$ and $ A=\delta_j$. Then we have
$$(\pi\otimes \pi)(\p_{\ell_j}\otimes \p_{\ell_i})\big(1\otimes \wvec{t_1,\cdots,t_m}{x_{\ell_{k+1}}, \cdots, x_{\ell_{k+m}}}\big)=0
 =(\delta_j\cks \delta_i)(\pi\otimes \pi)(1\otimes \wvec{t_1,\cdots,t_m}{x_{\ell_{k+1}}, \cdots, x_{\ell_{k+m}}}).
$$

By the definitions  of  $\p_{\ell_i}$  and  $\delta_i$, it is clear that
$$(\pi\otimes \pi)(B\otimes \p_{\ell_i})(\wvec{s_1,\cdots, s_k}{x_{\ell_1}, \cdots, x_{\ell_k}}\otimes 1)=0 =(A\cks \delta_i)(\pi\otimes \pi)(\wvec{s_1,\cdots, s_k}{x_{\ell_1}, \cdots, x_{\ell_k}}\otimes 1).
$$

By Lemma \mref{lem:partialanddelta},
$$\pi B(\wvec{s_1,\cdots, s_k}{x_{\ell_1}, \cdots, x_{\ell_k}})=A  \pi(\wvec{s_1,\cdots, s_k}{x_{\ell_1}, \cdots, x_{\ell_k}}).
$$
So
\begin{align*}
(\pi\otimes \pi)(B\otimes \p_{\ell_i})(\wvec{s_1,\cdots, s_k}{x_{\ell_1}, \cdots, x_{\ell_k}}\otimes \wvec{t_1,\cdots,t_m}{x_{\ell_{k+1}}, \cdots, x_{\ell_{k+m}}})=&(\pi\otimes \pi)\big(B(\wvec{s_1,\cdots, s_k}{x_{\ell_1}, \cdots, x_{\ell_k}})\otimes \p_{\ell_i}(\wvec{t_1,\cdots,t_m}{x_{\ell_{k+1}}, \cdots, x_{\ell_{k+m}}})\big)\\
=&A([s_1,\cdots, s_k])\otimes \pi\p_{\ell_i}(\wvec{t_1,\cdots,t_m}{x_{\ell_{k+1}}, \cdots, x_{\ell_{k+m}}}).
\end{align*}
Let $p_1=\ell_{k+1}$, $\ldots$, $p_m=\ell_{m+k}$, that is $\ell_i=p_{i-k}$. Then
$$(\pi\otimes \pi)(B\otimes \p_{\ell_i})(\wvec{s_1,\cdots, s_k}{x_{\ell_1}, \cdots, x_{\ell_k}}\otimes \wvec{t_1,\cdots,t_m}{x_{\ell_{k+1}}, \cdots, x_{\ell_{k+m}}})\\
=A([s_1,\cdots, s_k])\otimes \pi\p_{p_{i-k}}(\wvec{t_1,\cdots,t_m}{x_{p_{1}}, \cdots, x_{p_{m}}}).
$$
By  Lemma  \mref{lem:nodependchoice} and a change of variables,
\begin{align*}
 (A\cks \delta_i)(\pi\otimes \pi)(\wvec{s_1,\cdots, s_k}{x_{\ell_1}, \cdots, x_{\ell_k}}\otimes \wvec{t_1,\cdots,t_m}{x_{\ell_{k+1}}, \cdots, x_{\ell_{k+m}}})=&A([s_1,\cdots, s_k])\otimes \delta_{i-k}\pi(\wvec{t_1,\cdots,t_m}{x_{\ell_{k+1}}, \cdots, x_{\ell_{k+m}}})\\
=&A([s_1,\cdots, s_k])\otimes \delta_{i-k}\pi(\wvec{t_1,\cdots,t_m}{x_{p_{1}}, \cdots, x_{p_{m}}}).
\end{align*}
By  Lemma \mref{lem:partialanddelta},  we  have
\begin{align*}
  &(\pi\otimes \pi)(B\otimes \p_{\ell_i})(\wvec{s_1,\cdots, s_k}{x_{\ell_1}, \cdots, x_{\ell_k}}\otimes \wvec{t_1,\cdots,t_m}{x_{\ell_{k+1}}, \cdots, x_{\ell_{k+m}}})=(A\cks \delta_i)(\pi\otimes \pi)(\wvec{s_1,\cdots, s_k}{x_{\ell_1}, \cdots, x_{\ell_k}}\otimes \wvec{t_1,\cdots,t_m}{x_{\ell_{k+1}}, \cdots, x_{\ell_{k+m}}}),
  \end{align*}
as needed.
\end{proof}

\begin{lemma}
\mlabel{lem:deltacommuteDeltage1}
The coproduct $\descop: \calhd\longrightarrow \calhd\otimes \calhd$ is a coderivation:
$$ \descop \delta_i=(\id\cks \delta_i+\delta_i\otimes \id)\descop.
$$
\end{lemma}

\begin{proof}
It is obvious that $(\id\cks \delta_i+\delta_i\otimes \id)\descop({\bf 1}) =0=\descop \delta_i({\bf 1})$. Next let $(s_1,\cdots, s_k)\in \Z_{\ge 1}^k$.

{\bf Case 1:} If $1\le i\leq k$, then for distinct $\ell_1, \cdots, \ell_k \in \Z _{\ge  1}$, by  Lemma  \mref{lem:tensorpartialanddelta} and Eq.(\ref{eq:DeltaDer}), we have
\begin{align*}
\descop \delta_i([s_1,\cdots, s_k])
=&\sum_{j=1}^i s_j\descop([s_1,\cdots,s_j+1,\cdots, s_k])\\
=&\sum_{j=1}^i s_j(\pi\otimes \pi)\fraccop(\wvec{s_1,\cdots,s_j+1,\cdots, s_k}{x_{\ell_1},\cdots, x_{\ell_j},\cdots, x_{\ell_k}})\\
=&(\pi\otimes \pi)\fraccop \p_{\ell_i}(\wvec{s_1,\cdots,s_i,\cdots, s_k}{x_{\ell_1},\cdots, x_{\ell_i},\cdots, x_{\ell_k}})\\
=&(\pi\otimes \pi)(\id\otimes \p_{\ell_i}+\p_{\ell_i}\otimes \id)\fraccop(\wvec{s_1,\cdots,s_i,\cdots, s_k}{x_{\ell_1},\cdots, x_{\ell_i},\cdots, x_{\ell_k}})\\
=&(\id\cks \delta_i+\delta_i\otimes \id)(\pi\otimes \pi)\fraccop(\wvec{s_1,\cdots,s_i,\cdots, s_k}{x_{\ell_1},\cdots, x_{\ell_i},\cdots, x_{\ell_k}})\\
=&(\id\cks \delta_i+\delta_i\otimes \id)\descop ([s_1,\cdots, s_k]).
\end{align*}
{\bf Case 2:} If  $i>k$, then $\descop \delta_i([s_1,\cdots, s_k])=0$. By  Lemma  \mref{lem:tensorpartialanddelta} and Eq.(\ref{eq:DeltaDer}), we have
\begin{align*}
(\id\cks \delta_i+\delta_i\otimes \id)\descop([s_1,\cdots, s_k])=&(\id\cks \delta_i+\delta_i\otimes \id)(\pi\otimes \pi)\fraccop(\wvec{s_1,\cdots,s_k}{x_{\ell_1},\cdots,x_{\ell_k}})\\
=&(\pi\otimes \pi)(\id\otimes \p_{\ell_i}+\p_{\ell_i}\otimes \id)\fraccop(\wvec{s_1,\cdots,s_k}{x_{\ell_1},\cdots,x_{\ell_k}})\\
=&(\pi\otimes \pi)\fraccop\p_{\ell_i}(\wvec{s_1,\cdots,s_k}{x_{\ell_1},\cdots,x_{\ell_k}})\\
=&0.
\end{align*}
This completes the proof.
\end{proof}

Now we can complete the proof of Theorem~\mref{thm:shuhopf}.
\begin{prop}
The quintuple $(\calhd, \shapt, {\bf 1}, \reccop, \vep_{\geq 1})$ in Theorem~\mref{thm:shuhopf} coincides with the quintuple $(\calhd, \shapt, {\bf 1}, \descop, \vep_{\ge 1})$	in Theorem~\mref{t:deschopf} and hence is a Hopf algebra.
\mlabel{p:copcompatible}	
\end{prop}

\begin{proof}
By Lemma~\mref{lem:deltacommuteDeltage1}, $\descop$ satisfies the boundary conditions and the recursion that define $\reccop$ in Definition~\mref{d:reccop}. Thus the two coproducts are the same. Thus statement of the proposition holds and  $(\calhd, \shapt, {\bf 1}, \reccop, \vep_{\geq 1})$ is a Hopf algebra.
\end{proof}

\noindent
{\bf Acknowledgments.} This research is partially supported by
the NNSFC (12126354, 12471062).

\noindent
{\bf Declaration of interests. } The authors have no conflict of interest to declare that are relevant to this article.

\noindent
{\bf Data availability. } Data sharing is not applicable to this article as no data were created or analyzed.

\end{document}